\documentclass[a4paper,11pt]{article}

	\usepackage{enumerate,color,siunitx,enumitem}
	\usepackage{cite}
	\usepackage{bm,amsmath,amssymb,amsthm,mathtools,physics}
		\usepackage[symbol]{footmisc}
	
 \usepackage{tikz}
\usetikzlibrary{shapes.geometric, arrows, positioning}

\tikzstyle{startstop} = [rectangle, rounded corners, minimum width=5cm, minimum height=1cm, text centered, draw=black, fill=gray!20]
\tikzstyle{process} = [rectangle, minimum width=5cm, minimum height=1cm, text centered, draw=black, fill=orange!20]
\tikzstyle{decision} = [diamond, aspect=2, minimum width=5cm, minimum height=1cm, text centered, draw=black, fill=green!20]
\tikzstyle{io} = [trapezium, trapezium left angle=70, trapezium right angle=110, minimum width=5cm, minimum height=1cm, text centered, draw=black, fill=blue!20]
\tikzstyle{arrow} = [thick,->,>=stealth]

    \usepackage{pgfplots}
    \usepackage{placeins}

	\usepackage[margin=1in]{geometry}
	\usepackage{float,pdflscape,authblk,setspace,lineno}
     \usepackage{hyperref}
    \usepackage[capitalize]{cleveref}

	\usepackage{algorithm,algpseudocode}
	\numberwithin{equation}{section}
	\newtheorem{theorem}{Theorem}[section]
	
	\newtheorem{proposition}[theorem]{Proposition}
	\newtheorem{definition}[theorem]{Definition}
	
	\newtheorem{example}{Example}[section]
	\newtheorem{corollary}[theorem]{Corollary}
	\newtheorem{remark}[theorem]{Remark}
    
	\usepackage[font=small,labelfont=bf,figurename=Figure,%
		labelsep=period]{caption}
	\usepackage{graphicx}
    \usepackage{subcaption}

	\newenvironment{keywords}
    {\footnotesize \begin{center}\textbf{Keywords}\end{center}\vspace{-0.2cm}}
    {}

    \newcommand{\killpunct}[1]{}

 \title{Structural identifiability of linear-in-parameter parabolic PDEs through auxiliary elliptic operators}

	\date{\today}

	\author[1*]{Yurij Salmaniw}
	\author[1,2*]{Alexander P Browning}
	
	\affil[1]{Mathematical Institute, University of Oxford, Oxford, United Kingdom}
    \affil[2]{School of Mathematics and Statistics, The University of Melbourne, Melbourne, Australia}

	\footnotetext[1]{salmaniw@maths.ox.ac.uk and alex.browning@unimelb.edu.au}

\begin{document}

\maketitle

	\vfill
	\renewcommand{\abstractname}{Abstract}
	\begin{abstract}\noindent
        Parameter identifiability is often requisite to the effective application of mathematical models in the interpretation of biological data, however theory applicable to the study of partial differential equations remains limited. We present a new approach to structural identifiability analysis of fully observed parabolic equations that are linear in their parameters. Our approach frames identifiability as an existence and uniqueness problem in a closely related elliptic equation and draws, for homogeneous equations, on the well-known Fredholm alternative to establish unconditional identifiability, and cases where specific choices of initial and boundary conditions lead to non-identifiability. While in some sense pathological, we demonstrate that this loss of structural identifiability has ramifications for practical identifiability; important particularly for spatial problems, where the initial condition is often limited by experimental constraints. For cases with nonlinear reaction terms, uniqueness of solutions to the auxiliary elliptic equation corresponds to identifiability, often leading to unconditional global identifiability under mild assumptions. We present analysis for a suite of simple scalar models with various boundary conditions that include linear (exponential) and nonlinear (logistic) source terms, and a special case of a two-species cell motility model. We conclude by discussing how this new perspective enables well-developed analysis tools to advance the developing theory underlying structural identifiability of partial differential equations.
	\end{abstract}
	\vfill
 \vfill

\clearpage
\section{Introduction}

Mathematical models are now routine in the interpretation of biological data. Model parameters, in particular, provide an objective means of characterising behaviours that cannot be directly measured \cite{Liepe:2014,Gabor:2015dr}. Effective application of models for such characterisation, in addition to data extrapolation and prediction, often requires that parameters are identifiable \cite{Cobelli:1980,Raue:2009,Chis:2011}. Specifically, the \textit{structural identifiability} of a model is an assessment of whether model parameters can ever be identified from a given model and observation structure \cite{Bellman.1970,Cobelli:1980,Walter:1981,Walter:1987}.

Mathematically, the question of structural identifiability is to ask whether the map from model parameters to model outputs is injective \cite{Walter:1987,Renardy.2022}. If this injectivity holds only within a small neighbourhood of any given set of parameters, the model may be classified as \textit{locally structurally identifiable}, else the terminology \textit{globally structurally identifiable} is commonly used \cite{Xia.2003,Renardy.2022}. Succinctly, models are classified as non-identifiable if multiple distinct parameter sets give rise to indistinguishable model outputs (\cref{fig1}). For example, the reaction-diffusion equation \label{refc1}
    \begin{equation}
       u_t = d_1 u_{xx} + c_1 u,\quad u(0,t) = u(1,t) = 0,
    \end{equation}
admits the solution 
\begin{equation}\label{eq:linear_rd_sol}
    u(x,t) = \mathrm{e}^{\mu(c_1,d_1) t} \sin(\pi x),\quad \mu(c_1,d_1) = c_1 - d_1 \pi^2,
\end{equation}
where $d_1$ and $c_1$ denote the diffusivity and exponential growth rate, respectively.  As all parameter sets that satisfy $\mu(c_1,d_1) = \text{const.}$ give rise to an identical solution (see \cref{fig1}b), we classify such a model as non-identifiable from the initial condition $u(x,0) = \sin(\pi x)$. Tools for the assessment of structural identifiability are now well established for ordinary differential equation (ODE) models \cite{Bellu:2007,Raue:2009,Chis:2011,Barreiro.2023hg}, and are, perhaps, most pertinent in scenarios where only partial observations of the state space are made; for example, in a biochemical system where only one of a number of reactants is experimentally observable \cite{Bellman.1970}. Methods include those based on differential algebra, in which assessment is drawn from a so-called input-output relation: a higher order set of differential equations treated as polynomial in the derivatives of the observed state space \cite{Ljung:1994,Margaria.2001,Bellu:2007,Meshkat:2014}. Recent decades have seen a proliferation of methodologies based on differential-algebra \cite{Meshkat:2014}, as well as Taylor series, generating series and Lie derivatives \cite{Chis:2011,Ligon.2017qw}, and similarity transforms \cite{Vajda:1989,Chis:2011}.

\begin{figure}[!b]
	\centering
	\includegraphics{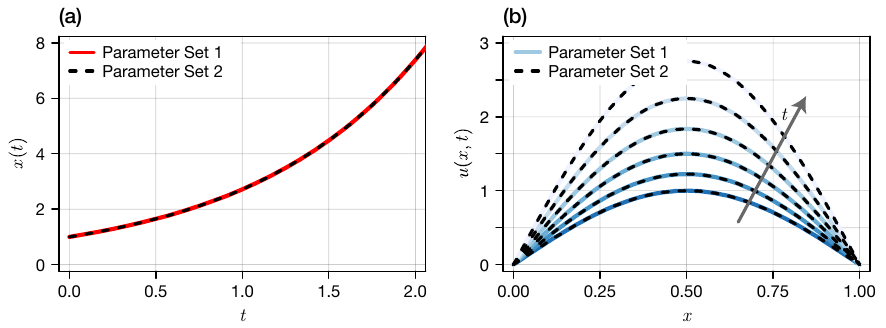}
	\caption[Figure 1]{Structural non-identifiability in a (a) ordinary- and (b) partial differential equation. Shown are indistinguishable solutions from two distinct parameter sets. In (a) we show the first component of $X$ subject to $\dot X = M_1 X$ (red solid) and $\dot X = M_2 X$ (black dashed) where $M_1 = (\begin{smallmatrix}1 & 0\\ 0 & 1\end{smallmatrix})$,  $M_2 = (\begin{smallmatrix}2 & -1\\ 1 & 0\end{smallmatrix})$, and $X(0) = (1,1)^\intercal$. In (b), we show the solution a linear reaction-diffusion equation (given by \cref{eq:linear_rd_sol}) at $t \in \{0,0.4,0.8,1.2,1.6,2\}$ for $(c_1,d_1) = (1,0.05)$ (blue solid) and $(c_1,d_1) \approx (2,0.15)$ (black dashed). Full details are given in Sections \ref{secODEs} and \ref{dirichlet} for the ODE and PDE models, respectively.}
	\label{fig1}
\end{figure}

Common to all general methodologies is that assessment is made based on the model structure itself, such that the existence or availability of an analytical solution is not required. Furthermore, questions of structural identifiability are typically assessed without regard to specific initial conditions or specific regions of parameter space: it is theoretically possible for structurally identifiable models to lose this property under certain constrains \cite{Saccomani.2003}. A trivial example is to consider a system initialised at equilibrium, such that the temporal dynamics are not observed. The literature draws for ODE models a clear distinction between so-called \textit{single-} and \textit{multi-experiment} identifiability, the latter in reference to the case that multiple distinct initial conditions are required to gain the ability to distinguish between parameter regimes \cite{Ligon.2017qw,Ovchinnikov.2022}. While there exists only very limited theory linking the question of structural identifiability to that of \textit{practical identifiability}---which asks whether models are identifiable from a finite and noisy set of experimental data---one might expect that experiments initiated near regimes that lose structural identifiability would yield comparatively less information.

Despite an increasing prevalence of spatiotemporal data, tools and associated theory for the assessment of structural identifiability in partial differential equation (PDE) models have only very recently become available \cite{Renardy.2022,Browning.2024d0n,Ciocanel.2024,Byrne.2025}. In our recent work \cite{Browning.2024d0n}, we extend the differential algebra approach to a relatively general system of linear reaction-diffusion equations. This approach converts an infinite-dimensional PDE into an algebraic system of equations, where each distinct component of the equation is treated as its own variable (e.g., $u =: p_1$, $\Delta u =: p_2$, etc.) This is a relatively accessible approach; however, it requires linear independence between these new variables, which may not always hold. Interestingly, this loss of linear independence (i.e., through eigenvalues and eigenfunctions of an operator) is precisely what allows us to construct solutions such as \eqref{eq:linear_rd_sol}. More recent work by Byrne et al. \cite{Byrne.2025} provides an extended framework for the assessment of structural identifiability in a more general class that includes nonlinear PDE models. There remains, however, little theory underpinning structural identifiability analysis of PDE models. In particular, analogies between highly studied properties of PDE models---such as the existence and uniqueness of solutions---and identifiability have not yet been drawn. In contrast, methodology developments for ODE are in no small part attributable to more generally applicable theory: Lie algebras, derivatives, and symmetries, for example \cite{Chis:2011}. In the present work, we seek to develop an approach that is accessible from an analytical point of view, and can be utilised with a graduate-level understanding of modern PDE theory.

Many questions relating to the identifiability of PDE models fundamentally differ from that of ODE models. While observations of ODE models are typically assumed to be temporal observations of one or more system states (or linear combinations thereof), it is not always the case that observations of spatial processes are themselves spatial. For example, many routine cell biology experiments are inherently two-dimensional, but reported observations are typically either scalar or otherwise spatially averaged (cell counts or density profiles, for example) \cite{McInerney.2019rs,Simpson:2020}. Furthermore, several considerations relating to the initial condition are unique to the spatial problem. In our previous work, we demonstrate that even specifying a parametric form for the observed initial condition within a partially observed spatial system can have unintended consequences for conclusions relating to identifiability \cite{Browning.2024d0n}. Finally, experiments that yield spatiotemporal data are often much more constrained than for those interpretable using ODE models: experimentalists are often practically limited to a very small number of spatially inhomogeneous initial conditions. Altogether, consideration of the initial condition and the constraints thus imposed is of high importance for the identifiability analysis of PDE models. 

In this work, we deviate from the more general techniques considered in previous work \cite{Browning.2024d0n,Byrne.2025} to study identifiability in a class of parabolic PDE models subject to one or a small number of initial conditions. Our approach is deliberately aligned with highly developed existing theory for analysis of PDEs. Specifically, we phrase the question of identifiability as an existence and uniqueness problem in a closely related elliptic PDE. We consider homogeneous parabolic equations of the form 
    \begin{align}\label{eq:pde_general}
        u_t = L u + f(x,u;B),
    \end{align}
in some prescribed spatial domain $\Omega \subset \mathbb{R}^n$ and subject to a specified boundary condition $\mathcal{B}u = 0$, where $L = L[A_1] (\cdot)$ is a linear, second order, uniformly elliptic operator depending on a parameter point $A_1$, and $f(x, u) = f(x, u; B_1)$ satisfies $f(x,0; B_1 ) = 0$ for all possible parameter points $B_1$. We always assume that the domain $\Omega$ has a smooth (e.g., $C^2$) boundary to avoid technicalities with the regularity of the solutions obtained. However, we note that this restriction can be weakened significantly while retaining the approach's validity. The boundary condition $\mathcal{B}$ may include the homogeneous Dirichlet, Neumann, or Robin conditions or periodic conditions; see Section \ref{sec:homogeneous_pde}. We note, however, that this approach is not limited to these choices. We assume that the boundary operator does not depend on an unknown parameter for simplicity. The parameter point $A_1$ may incorporate, for example, diffusion rates, advection rates, or growth rates. The parameter point $B_1$ may include growth rates, competition rates, or Allee effects; in both cases, the specific parameters included in the points $A_1$ and $B_1$ depend on the model under consideration. \label{refc2}

The only other requirement is that $L$ and $f$ are linear in their parameters in the sense that
	\begin{equation*}
		L[ A_1 ]u + L [ A_2 ]u = L[ A_1 + A_2 ] u \, ; \quad \quad f(x,  u;  B_1 ) + f( x,  u; B_2 ) = f( x,  u ; B_1 + B_2),
	\end{equation*}
for any parameter points $A_i$, $B_i$, $i=1,2$,  for all inputs $(x,u)$. To be clear, we do not require that models are also linear in the state variable. For example, the  classical  logistic growth model with $ f(x,u;B_1) = f(u, B_1  ) := a_1 u - b_1 u^2$ where $B_1 = (a_1,b_1)  \in \mathbb{R}^+ \times \mathbb{R}^+  $ can be studied using our approach.   We highlight that models of the form \cref{eq:pde_general} are routinely used in the interpretation of biological data; for example, in the context of cell proliferation and migration \cite{Murray:2002,Maini.2004rud,Swanson.2011}. 

At the most fundamental level, the approach we present here is to assume that there are two distinct solutions for a fixed set of parameters, and then to derive a contradiction (see Figure \ref{fig2}). The challenge is to determine when such a contradiction can be obtained. The rest of the paper is dedicated to developing this procedure. We begin in \Cref{secODEs} by drawing an analogy to linear algebra to demonstrate structural identifiability in linear systems of ODEs with constant coefficients. In \Cref{seclPDEs} we extend this analysis to linear PDE models of the form given in \cref{eq:pde_general} with $f(x,u ; B ) \equiv 0$. It is in \Cref{sec:general_framework} that we present formal definitions for what we term \textit{parameter distinguishability} and \textit{model identifiability}, and in which we outline the methodology to establish identifiability using properties of the elliptic operator $L$. This methodology is applied to the homogeneous linear parabolic PDE model in \cref{dirichlet,otherbc} for Dirichlet, Neumann, Robin, and periodic boundary conditions. Next, we extend the analysis to scalar models with state nonlinearity in $f(x,u ; B )$ and systems of equations in \cref{sec:inhomogeneous_pde}. Finally, we demonstrate the ramifications of our analysis on the practical identifiability of the linear model in \cref{sec:practical} before concluding with a discussion and outlining future research directions in \cref{sec:discussion}.

\section{Identifiability of fully observed systems of linear ODEs}\label{secODEs}

We begin by considering the structural identifiability of a fully observed system of ODEs subject to a single initial condition. While straightforward, this example carries two clear analogues with the PDE problem. The first is direct and arises by considering that a finite-difference representation of a linear PDE will result in a corresponding system of linear ODEs. We draw more closely on the second in \Cref{seclPDEs}, by viewing the coefficient matrix that characterises the ODE system more generally as an elliptic linear operator.

Consider an $n$-dimensional, linear, homogeneous ODE system of the form
\begin{align}\label{eq:linear_homogeneous_ODE_proto}
    \begin{cases}
        \dot{X} = M_1 X, \quad t > 0, \cr
        X(0) = X_0,
    \end{cases}
\end{align}
where $X = (x_1(t), x_2(t), \ldots, x_n(t))^\intercal$, $X_0 \in \mathbb{R}^n$ is an initial condition, and $M_1 \in \mathbb{R}^{n \times n}$ is an $n\times n$ matrix of real-valued coefficients  which we refer to as a \textit{parameter point} in the space $\mathbb{R}^{n \times n}$; see Definition \ref{def:parameters} for a precise definition in the context of PDE identifiability.  The matrix $M_1$ contains (up to) $n^2$ parameters that we wish to identify from a given observation of the trajectory $X(t)$. We are concerned with the following question: when is it the case that there exists another matrix of coefficients, say $M_2 \in \mathbb{R}^{n \times n}$, such that $X(t)$ satisfies simultaneously $\dot{X} = M_1 X$ and $\dot{X} = M_2 X$? If such a matrix exists, we refer to $M_1$ and $M_2$ as \textit{indistinguishable} with respect to $X_0$. We then conclude that model \eqref{eq:linear_homogeneous_ODE_proto} is non-identifiable from the initial condition $X_0$. We define precisely indistinguishability for PDE models in Definition \ref{def:parameters}. First, we consider the following Proposition.

\begin{proposition}\label{prop:ODE_prop_1}
    Fix $M_1 \in \mathbb{R}^{n \times n}$ and suppose $M \in \mathbb{R}^{n \times n}$ satisfies the following:
\begin{enumerate}[label=\roman*.)]
    \item $M$ is singular and nontrivial;
    \item $M$ commutes with $M_1$, that is, $M_1 M = M M_1$.
\end{enumerate}
Then, for any initial condition $X_0 \in \ker (M)$, $X(t)$ simultaneously solves \eqref{eq:linear_homogeneous_ODE_proto} with  parameter point  $M_1$ and
\begin{align}
    \begin{cases}
        \dot{X} = M_2 X, \quad t>0, \nonumber \\
        X(0) = X_0,
    \end{cases}
\end{align}
where $M_2 := M_1 + M$. In particular, the matrix $M_2$ encodes possible parameter point combinations that are indistinguishable from the point $M_1$ with respect to the initial condition $X_0$.
\end{proposition}
\begin{proof}
    The fundamental solution of problem \eqref{eq:linear_homogeneous_ODE_proto} can be written in its matrix-exponential form:
    $$
X(t) = \mathrm{e}^{M_1 t} X_0 .
    $$
    We may differentiate with  respect  to $t$ and perform some elementary calculations to obtain:
    \begin{align}
        \dot{X} = M_1 \mathrm{e}^{M_1 t} X_0 & =  (M_1 + M ) \mathrm{e}^{M_1 t} X_0 - M \mathrm{e}^{M_1 t} X_0 \nonumber \\
        & =  (M_1 + M) X - \mathrm{e}^{M_1 t} M X_0 \nonumber \\
        & =  (M_1 + M ) X,
    \end{align}
    where the second line is obtained using the commutativity of $M_1$ and $M$, and the third line uses that $X_0$ belongs to the kernel of $M$.
\end{proof}

\begin{remark}
    We note that necessary and sufficient conditions for the structural identifiability of model \eqref{eq:linear_homogeneous_ODE_proto} have been established; see \cite{Stanhope2014}.
\end{remark}

The following is an elementary example to demonstrate Proposition \eqref{prop:ODE_prop_1}.

\begin{example}\label{example_1}
    Suppose $n=2$ and that $X$ solves \eqref{eq:linear_homogeneous_ODE_proto} with $M_1$ given by
    $$
    M_1 := \begin{pmatrix}
        2 & 3 \\ 1 & 4
    \end{pmatrix}
    $$
    and the initial condition $X_0$ is to be determined. The general solution is given by $x_1(t) = c_1 \mathrm{e}^{5t} - 3 c_2 \mathrm{e}^t$ and $x_2(t) = c_1 \mathrm{e}^{5t} + c_2 \mathrm{e}^t$, where $c_1$, $c_2$ are constants determined by the initial condition.
    
Then, the set of matrices that commute with $M_1$ is given precisely by
    $$
    M := \begin{pmatrix}
        a & 3b \\ b & a + 2b 
    \end{pmatrix},
    $$ 
   for any $a,b \in \mathbb{R}$. To ensure $M$ is singular, we restrict $a,b$ such that $a(a + 2b) - 3b^2 = 0$. If we fix $x_2(0) \in \mathbb{R}$ and choose $x_1(0) = - \tfrac{3b}{a} x_2(0)$, the singularity of $M$ ensures that
   $$
\begin{pmatrix}
    x_1(0) \\ x_2(0)
\end{pmatrix} :=
\begin{pmatrix}
    - \tfrac{3b}{a} \\ 1
\end{pmatrix} x_2(0)
   $$
   belongs to the kernel of $M$. Hence, Proposition \ref{prop:ODE_prop_1} applies and $X(t)$ solves both
   $\dot{X} = M_1 X$ and $\dot{X} = M_2 X = (M_1 + M)X$. 
\end{example}

In the procedure above, we have effectively reduced the question of structural identifiability of linear constant-coefficient ODE models to an algebraic problem  related  to whether the initial condition is in the kernel of some matrix. This leads to distinct  parameter points , $M_1$ and $M_2 := M_1 + M$, which produce identical output in the form of $X(t)$, and hence are indistinguishable in the sense that perfect observation of the solution trajectory $X(t)$ from $X_0$ does not allow us to distinguish between the  points  $M_1$ and $M_1 + M$. We provide an additional example of non-identifiability in two-dimensional ODE models in \cref{fig1}a.

This exercise motivates a generalisation to infinite-dimensional PDE models that depend linearly on their parameters. In particular, given a linear parabolic PDE in standard form (i.e., so that the coefficient of the time derivative is set to $1$), we can reduce the question of \textit{distinguishability of parameters in a parabolic model} into a question of \textit{existence and uniqueness of solutions to an auxiliary elliptic equation}. The existence and uniqueness of solutions to this new problem provide significant information on whether parameters are or are not structurally identifiable.  In fact, just as we lose distinguishability in the ODE system when $X_0$ lies in the kernel of the matrix $M$, so too can we lose distinguishability in the PDE setting when the initial condition $u_0(x)$ belongs to the kernel of a differential operator $L$—a direct analogy that informs our subsequent approach. This is precisely the analogy we pursue in Section 3, where we show how parameter distinguishability in a parabolic PDE depends on the existence of nontrivial kernels of an associated elliptic operator.  We begin with the case of homogeneous parabolic equations.

\section{Identifiability of linear parabolic PDE models}\label{seclPDEs}

\subsection{The general framework}\label{sec:general_framework}

Motivated by the example shown in Section \ref{secODEs}, we adapt this approach to the analysis of homogeneous parabolic PDEs that are linear in their parameters. More precisely, we consider a model of the form
\begin{align}\label{eq:pde_general2}
\begin{cases}
        u_t = L u , \hspace{2.15cm} \text{ for } \quad (x,t) \in \Omega \times (0,T), \cr
        u(x,0) = u_0 (x), \hspace{1cm} \text{ for } \quad x \in \overline{\Omega}
\end{cases}
\end{align}
in some prescribed spatial domain $\Omega \subset \mathbb{R}^n$, where $L = L[A_1] (\cdot)$ is a linear, second order, uniformly elliptic operator depending on a parameter point $A_1 \in \mathbb{R}^{n+2}$, and $T>0$ is some terminal time.  We consider only bounded domains with smooth boundary $\partial \Omega$ for simplicity while noting that this is not a prerequisite to apply this approach.  We restrict the parameters of $L$ to $(n+2)$ dimensions, as this fully encodes a second order linear operator with constant linear isotropic diffusion $d \in (0,\infty)$ (one-dimensional), drift $b \in \mathbb{R}^n$ ($n$-dimensional), and/or low-order terms $c \in \mathbb{R}$ (one-dimensional), see \cref{elliptic_operator}. The following terminology will be useful.

\begin{definition}\label{def:parameters}
We define a \textup{parameter point} as any element in $(n+2)$-dimensional space, distinguishing this from a \textup{parameter set}, which denotes a subset of $(n+2)$-dimensional space consisting of multiple parameter points. The term \textup{parameter domain} refers to the valid region within parameter space, as determined by specific model constraints.
\end{definition} 
The concept of \textit{parameter domain} is most readily understood when one considers diffusion rates, which are assumed nonnegative. This will remove some redundancy of parameters through, e.g., symmetry of the equation in subsequent arguments. In what follows, we will identify all distinguishable and indistinguishable points for linear and nonlinear scalar equations, and therefore subsets of the parameter domain for which models are or are not structurally identifiable, see Proposition \ref{prop:distinguish_identifiable}.

As will become apparent, parameter distinguishability, and therefore structural identifiability, will depend intimately on the particular initial condition of the problem. Therefore, we will henceforth denote by
$$
u(x,t) = X[A_1,u_0] (x,t)
$$ 
the flow generated by model \eqref{eq:pde_general2} corresponding to the parameter point $A_1 = (d_1, b_1, c_1)$ and initial condition $u_0(x)$. Strictly speaking, one must equip problem \eqref{eq:pde_general2} with, e.g., a valid boundary operator so that the problem is well-posed. To maintain the current focus, we will save such technical details for Section \ref{sec:homogeneous_pde} and simply refer to the solution of problem \eqref{eq:pde_general2} in the definitions of distinguishability and identifiability to follow. Given two parameter points $A_1$, $A_2$, due to symmetry we will without loss of generality always assume that $d_1 \geq d_2$ (otherwise, reverse the roles of $A_2$ and $A_1$). 

This motivates the following definition of \textit{indistinguishable} parameter points, and naturally \textit{distinguishable} and \textit{unconditionally distinguishable} parameter points.
\begin{definition}\label{def:distinguish}
    We call two distinct parameter points $A_1$, $A_2$ \textup{indistinguishable with respect to $u_0$} if $X[ A_1, u_0] (x,t) = X[ A_2, u_0] (x,t)$ for all $(x,t) \in \Omega \times (0,T)$. Otherwise, we say they are \textup{distinguishable with respect to $u_0$}. If two parameter points are distinguishable with respect to $u_0$ for any $u_0 \not\equiv 0$ we say they are \textup{unconditionally distinguishable}. By \textup{non-identifiable solution}, we refer to a solution $X[A,u_0](x,t)$ that can be obtained from indistinguishable parameter points.
\end{definition}


We now make explicit the connection between parameter distinguishability and model identifiability. Adhering to existing literature, we adopt the following definitions of identifiability \cite{Stanhope2014}. First, one may consider identifiability from \textit{some} initial condition.
\begin{definition}\label{def:identifiable_3}
    Model \eqref{eq:pde_general2} is identifiable in a subset $\Gamma \subset \mathbb{R}^{n+2}$ if and only if, for all $A_1$, $A_2$ belonging to $\Gamma$, $A_1 \neq A_2$, there exists $u_0$ such that $X[A_1, u_0] (\cdot,\cdot) \not\equiv X[A_2, u_0] (\cdot, \cdot)$.
\end{definition}
When we write $X[A_1, u_0] (\cdot,\cdot) \not\equiv X[A_2, u_0] (\cdot, \cdot)$ above, this is understood in the sense that there exists some $(x_0,t_0) \in \Omega \times (0,T)$ so that $X[A_1, u_0] (x_0,t_0) \neq X[A_2, u_0] (x_0,t_0)$.  In this case, one hopes to conclude that the space of possible initial conditions that lead to parameter identifiability is large.  This definition is not very restrictive, and in fact leads to identifiability in all models we consider here; however, this definition is often less desirable in practice as it implicitly assumes fine control over the initial condition, which is not often the case (for example, often only cell confluence, and not geometry, in two-dimensional cell migration scratch assays can be varied \cite{Jin:2016}).


For PDE models, it is appropriate to specify identifiability from \textit{a particular} initial condition as in Definition \ref{def:distinguish}.
\begin{definition}\label{def:identifiable_1}
    Model \eqref{eq:pde_general2} is identifiable in a subset $\Gamma \subset \mathbb{R}^{n+2}$ from $u_0$ if and only if, for all $A_1$, $A_2$ belonging to $\Gamma$, $A_1 \neq A_2$ implies that $X[A_1, u_0] (\cdot,\cdot) \not\equiv X[A_2, u_0] (\cdot, \cdot)$.
\end{definition}

When this holds for all nontrivial initial conditions, we may obtain \textit{unconditional identifiability}. 
\begin{definition}\label{def:identifiable_2}
    Model \eqref{eq:pde_general2} is unconditionally identifiable in $\Gamma \subset \mathbb{R}^{n+2}$ if and only if, for all $A_1$, $A_2$ belonging to $\Gamma$, $A_1 \neq A_2$ implies that for each $u_0 \not\equiv 0$, $X[A_1, u_0] (\cdot,\cdot) \not\equiv X[A_2, u_0] (\cdot, \cdot)$.
\end{definition}

{From the definitions above, we observe that concepts of identifiability relate closely to our concept of distinguishability. There is, however, a subtle difference: identifiability concerns entire subsets of parameter space for which solutions differ somewhere in spacetime, whereas distinguishability reduces to a comparison of two particular parameter points. More precisely, we can say the following, each of which follows directly from their respective definitions.
\begin{proposition}\label{prop:distinguish_identifiable}
    Define $\Theta = \Theta(u_0)$ to be the following collection of pairs of distinguishable parameter points for a fixed initial condition $u_0$:
    $$
\Theta (u_0) := \{ A_1, A_2 \in \mathbb{R}^{n+2} : A_1, A_2 \text{ are distinguishable with respect to } u_0 \}.
    $$
    Then the following hold.
    \begin{enumerate}[label=\roman*.)]
        \item If, for every pair of parameter points $A_1$, $A_2$ there exists a $\tilde u_0$ such that $A_1, A_2 \in \Theta (\tilde u_0)$, then model \eqref{eq:pde_general2} is identifiable in $\Gamma = \mathbb{R}^{n+2}$ in the sense of Definition \ref{def:identifiable_3}. Conversely, if model \eqref{eq:pde_general2} is identifiable in the sense of Definition \ref{def:identifiable_3}, then there always exists some $\tilde u_0$ such that any pair of points $A_1$, $A_2$ are distinguishable.
        \item Model \eqref{eq:pde_general2} is identifiable in $\Gamma = \Theta(u_0)$ in the sense of Definition \ref{def:identifiable_1}. Conversely, if model \eqref{eq:pde_general2} is identifiable from $u_0$ in the sense of Definition \ref{def:identifiable_1}, then any pair of points $A_1, A_2 \in \Gamma$ are distinguishable with respect to $u_0$.
        \item Define $\tilde \Gamma := \cap_{u_0} \Theta (u_0)$, where the intersection is taken over all compatible initial conditions for model \eqref{eq:pde_general2}. Then model \eqref{eq:pde_general2} is unconditionally identifiable in $\tilde \Gamma$ in the sense of Definition \ref{def:identifiable_2}. Conversely, if model \eqref{eq:pde_general2} is unconditionally identifiable in a subset $\Gamma$, then points belonging to $\Gamma$ are unconditionally distinguishable. 
    \end{enumerate}
\end{proposition}
 In Proposition \ref{prop:distinguish_identifiable} iii.), we refer to all ``compatible" initial conditions. This depends on the specific model considered and is left intentionally vague. For example, this should include nonnegativity if the quantity measured does not achieve negative values (e.g., a population density). It may also assume that the initial condition satisfies the prescribed boundary condition so that the solution is continuous up to time $t=0$. Generally speaking, compatibility here means the problem is well-posed for those initial conditions. \label{refc5}

 Definitions \ref{def:parameters}-\ref{def:identifiable_2} develop  into a useful distinction as it is instructive to first identify and describe distinguishable points, and then to describe entire sets of distinguishable points, leading to identifiability. Proposition \ref{prop:distinguish_identifiable} paired with results of Section \ref{sec:homogeneous_pde} will allow one to conclude the following. 
\begin{itemize}
    \item Model \eqref{eq:pde_general2} is always identifiable in the sense of Definition \ref{def:identifiable_3};
    \item Model \eqref{eq:pde_general2} is identifiable in $\Gamma = \mathbb{R}^{n+2}$ from a wide class of initial conditions in the sense of Definition \ref{def:identifiable_1}, and those initial conditions that lead to non-identifiability correspond to elements of the kernel of an auxiliary elliptic equation (see Theorem \ref{thm:main_theorem_1});
    \item Model \eqref{eq:pde_general2} is never unconditionally identifiable in $\Gamma = \mathbb{R}^{n+2}$ the sense of Definition \ref{def:identifiable_2}; equivalently, there always exists some initial condition for which model \eqref{eq:pde_general2} is not identifiable.
\end{itemize}
}


Consequently, we are less concerned about the subset of parameter space alone, and more concerned about how the initial and boundary conditions interact with the model parameters, leading to indistinguishable points. By looking at properties of the elliptic operator $L[A]$ from \eqref{eq:pde_general2}, and its degenerate counterpart $L_0[A]$ (see \cref{eq:degenerate_L}), as they depend on auxiliary parameter points $A$ (see Section \ref{sec:homogeneous_pde}), we can identify precisely all combinations of indistinguishable parameter points. We seek to describe this set of combinations of parameter points. \label{refc6}

Given a subset $\mathcal{A} \subset [0,\infty) \times \mathbb{R}^{n+1}$, it will be useful to introduce the following equivalence relation on $[0,\infty) \times \mathbb{R}^{n+1}$ between parameter points:
\begin{align}\label{eq:equiv_relation}
    A_1 \sim A_2 \quad \text{ if and only if } \quad A:= A_1 - A_2 \in \mathcal{A} .
\end{align}
In this way, we will derive a correspondence between \textit{indistinguishability of parameter points} and \textit{equivalence of parameter points} with respect to the equivalence relation \eqref{eq:equiv_relation}: if $A_1$ and $A_2$ are equivalent in the sense of \eqref{eq:equiv_relation}, they may be indistinguishable under certain conditions. If we can exhaustively identify this set $\mathcal{A}$, then the quotient space
\begin{align}\label{eq:identifiable_set}
    \mathcal{R} := [0,\infty) \times \mathbb{R}^{n+1} / \mathcal{A} 
\end{align}
will provide a description of pairs of parameter points which are guaranteed to be distinguishable. Note that $A_1$, $A_2$ belonging to $\mathcal{R}$ is equivalent to $A_1 \not\sim A_2$. We now explain how equivalence of parameter points may lead to indistinguishability, but not necessarily so. \\

Our approach can be broken into several key steps. A flow chart diagram for Steps 1-3 can be found in Figure \ref{fig2}. Suppose, for concreteness, that the operator $L$ depends on a parameter point $(a, b) \in \mathbb{R}^2$.
\begin{enumerate}[label=Step \arabic*:]
    \item Assume, to derive a possible contradiction, that there exists a smooth solution $u(x,t)$ solving  problem  \eqref{eq:pde_general2} for two distinct parameter points $A_1:= (a_1, b_1)$, $A_2:=(a_2, b_2)$. Take the difference between the two equations solved by $u$ and use the linearity-in-coefficients property to obtain an auxiliary \textit{elliptic} partial differential equation depending on the auxiliary parameter point $A := (a_1 - a_2, b_1 - b_2)$.
    \item Partition the auxiliary parameter domain into two subsets, which we denote by $\mathcal{A}$ and $\mathcal{R}$. $\mathcal{A}$ will contain all parameter points such that the auxiliary elliptic equation has a nontrivial solution; $\mathcal{R}$ will contain those points for which there is no nontrivial solution (zero is always a solution). The set $\mathcal{R}$ encodes parameter-point combinations that are distinguishable. The set $\mathcal{A}$ \textit{may} lead to indistinguishable parameter points, but not necessarily so. 
    \item Partition $\mathcal{A}$ once more into the sets $\mathcal{A}_{\textup{NI}}$ ($\textup{NI}$ for \textit{not identifiable}) and $\mathcal{A}_{\textup{I}}$ ($\textup{I}$ for \textit{identifiable}). $\mathcal{A}_{\textup{NI}}$ is comprised of parameter points from $\mathcal{A}$ for which there exists a solution solving the original time-dependent problem \textit{and} the auxiliary elliptic equation simultaneously; $\mathcal{A}_{\textup{I}}$ is then given by $\mathcal{A} \setminus \mathcal{A}_{\textup{NI}}$. Parameters belonging to $\mathcal{A}_{\textup{I}}$ are distinguishable; parameters belonging to $\mathcal{A}_{\textup{NI}}$ lead to indistinguishability. 
\end{enumerate}
At this stage, assuming one can execute each step of the procedure outlined above, we will have fully described the nature of identifiability for model \eqref{eq:pde_general2} via Proposition \ref{prop:distinguish_identifiable}. One may introduce an optional fourth step where one further subdivides $\mathcal{A}_{\textup{NI}}$ based on some practical considerations. For example, if one is measuring a population density, we expect that solutions will remain nonnegative. One might then subdivide $\mathcal{A}_{\textup{NI}}$ further as follows:
\begin{enumerate}[label=Step \arabic*:]
  \setcounter{enumi}{3}
  \item Partition $\mathcal{A}_{\textup{NI}}$ into $\mathcal{A}_{\textup{NI}}^+$ and $\mathcal{A}_{\textup{NI}}^-$, parameters for which solutions are nonnegative and those which are negative somewhere, respectively. 
\end{enumerate}
Due to the generality of our approach, we obtain many indistinguiable points stemming from a sign-changing initial condition; one may then safely ignore such cases via a compatibility condition, in this case, non-negativity. \\

\begin{figure}[!t]
\begin{center}
    \begin{tikzpicture}[node distance=1.5cm, scale=0.8, transform shape, every node/.style={font=\normalsize}]

\tikzstyle{startstop} = [rectangle, rounded corners, text centered, draw=black, fill=gray!20]
\tikzstyle{process} = [rectangle, rounded corners, text centered, draw=black, fill=blue!20]
\tikzstyle{decision} = [diamond, aspect=2, text centered, draw=black, fill=cyan!20]
\tikzstyle{arrow} = [->,>=stealth]

\node (start) [startstop, minimum width=4cm, text width=10cm] {Assume existence of smooth solution \( u \) solving \cref{eq:pde_general2} for two distinct parameter points \( A_1 \) and \( A_2 \)};

\node (process1) [process, below of=start, yshift=0.25cm, minimum width=5cm, text width=10cm] {Compute difference between the two equations};

\node (process2) [process, below of=process1, yshift=0.25cm, minimum width=5cm, text width=10cm] {Obtain auxiliary elliptic PDE with parameter point \( A = A_1 - A_2 \)};

\node (decision1) [decision, below of=process2, yshift=-1.0cm, minimum width=4cm, text width=3.5cm] {Does the auxiliary elliptic PDE have a nontrivial solution?};

\node (process3) [process, below of=decision1, xshift=-4cm, yshift=-1cm, minimum width=4.5cm, text width=4.5cm] {Parameter point $A$ belongs to \( \mathcal{A} \)};

\node (process4) [process, below of=decision1, xshift=4cm, yshift=-1cm, minimum width=4.5cm, text width=4.5cm] {Parameter point $A$ belongs to \( \mathcal{R} \) ($A_1, A_2$ are distinguishable)};

\node (decision2) [decision, below of=process3, xshift =4cm, yshift=-1cm, minimum width=4cm, text width=3.5cm] {Does \( u \) also solve the original time-dependent problem?};

\node (process5) [process, below of=decision2, xshift=-4cm, yshift=-1cm, minimum width=5cm, text width=4.5cm] {Parameter point $A$ belongs to \( \mathcal{A}_{\textup{NI}} \) \\ ($A_1, A_2$ are 
 indistinguishable)};

\node (process6) [process, below of=decision2, xshift=4cm, yshift=-1cm, minimum width=5cm, text width=4.5cm] {Parameter point $A$ belongs to \( \mathcal{A}_{\textup{I}} \) \\ ($A_1, A_2$ are 
 distinguishable)};

\draw [arrow] (start) -- (process1);
\draw [arrow] (process1) -- (process2);
\draw [arrow] (process2) -- (decision1);
\draw [arrow] (decision1) -| node[anchor=south east, xshift=-0.2cm, yshift=0.2cm] {Yes} (process3);
\draw [arrow] (decision1) -| node[anchor=south west, xshift=0.2cm, yshift=0.2cm] {No} (process4);
\draw [arrow] (process3) -- (decision2);
\draw [arrow] (decision2) -| node[anchor=south east, xshift=-0.2cm, yshift=0.2cm] {Yes} (process5);
\draw [arrow] (decision2) -| node[anchor=south west, xshift=0.2cm, yshift=0.2cm] {No} (process6);

\end{tikzpicture}
\end{center}
\caption{A flow chart diagram outlining Steps 1-3 described in Section \ref{sec:general_framework}.}\label{fig2}
\end{figure}
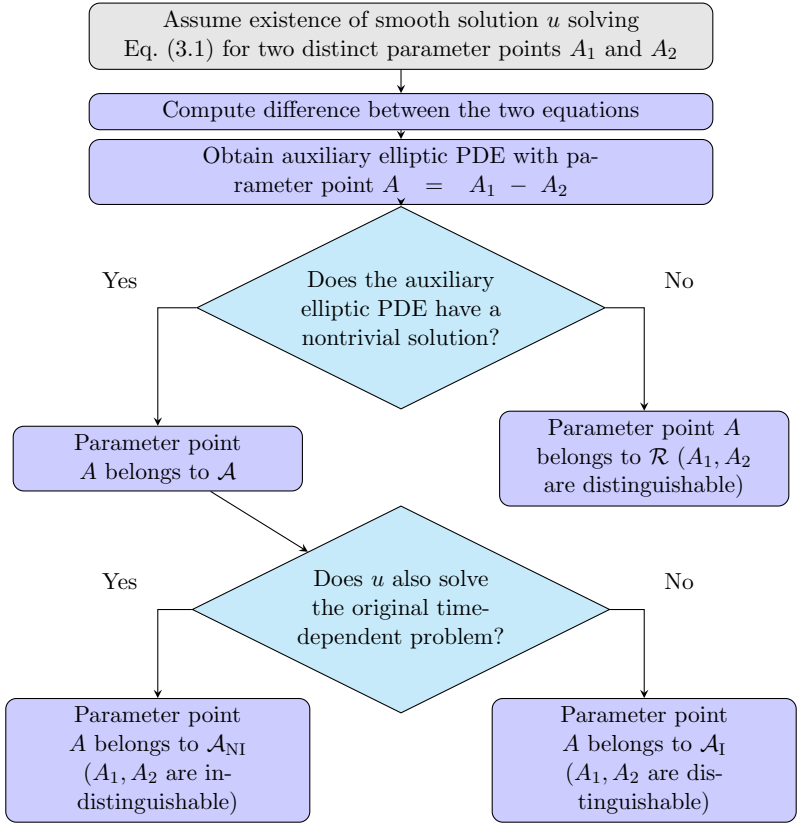

In fact, for a fixed boundary operator of Dirichlet, Neumann, Robin type, or periodic type, we can identify potentially indistinguishable parameter points by examining the set $\mathcal{A}$ given by
\begin{align}
    \mathcal{A} = \left\{ A \in [0,\infty) \times \mathbb{R}^{n+1} : \ker ( L[A] ) \neq \{ 0 \} \right\} ,
\end{align}
where $\ker ( L[A] )$ denotes the set of all functions $\psi$ belonging to the domain of $L$ such that $L[A] \psi = 0$, and $L$ is defined in \cref{elliptic_operator} below. $\mathcal{A}$ is often non-empty; we will provide examples to show this explicitly. Importantly, we emphasize that the kernel of $L$ depends intimately on which boundary condition is chosen. Equivalently, the boundary operator restricts the domain of definition for $L$ and, hence, restricts valid elements of the kernel. The domain of $L$ shrinks, therefore broadening the space of initial conditions for which model \eqref{eq:pde_general2} is identifiable in the sense of Definition \ref{def:identifiable_1}.

In the following subsections, we perform these steps in detail to expose how one can understand identifiability versus non-identifiability. In Section \ref{sec:homogeneous_pde}, we provide a complete description of identifiability for homogeneous models (i.e., $f \equiv 0$). As will be made clear, this depends closely on the boundary conditions of the PDE itself. Therefore, we treat in detail four cases one typically encounters: a homogeneous Dirichlet, Neumann, and Robin condition, as well as the periodic boundary condition. In Section \ref{sec:inhomogeneous_pde}, we then explore the case of a truly nonlinear right hand side that is linear in its parameters, focusing on the logistic model. We conclude Section \ref{sec:inhomogeneous_pde} by applying the results for the scalar homogeneous equation to construct nonidentifiable solutions to a two-species cell motility model.

\subsection{Linear, homogeneous parabolic equations}\label{sec:homogeneous_pde}

For convenience, we will denote by $L = L[A]$ the elliptic operator
\begin{align}\label{elliptic_operator}
    L[A] := d \Delta + b \cdot \nabla + c
\end{align}
for $A \in [0,\infty) \times \mathbb{R}^{n+1}$. For all $d>0$, $L[A]$ is a uniformly elliptic operator; when $d=0$, the operator $L$ degenerates and we write 
\begin{align}\label{eq:degenerate_L}
   L_0 [A] := b \cdot \nabla + c. 
\end{align} 
We also write $L_1[A]:= d \Delta + b \cdot \nabla$ to denote the operator without low order terms.

We begin with the following homogeneous equation with constant coefficients
\begin{align}\label{eq:mainPDE}
    \begin{cases}
        u_t = L[A_1]u := d_1 \Delta u + b_1 \cdot \nabla u + c_1 u,\quad  \text{ for } \quad (x,t) \in \Omega \times (0,T), \cr 
        \mathcal{B} u = 0, \hspace{3.1cm} \text{ for } \quad (x,t) \in \partial \Omega \times (0,T) , \cr
        u(x,0) = u_0(x), \hspace{1.75cm} \text{ for }\quad x \in \overline{\Omega},
    \end{cases}
\end{align}
in a smooth, bounded domain $\Omega \subset \mathbb{R}^n$. We denote by $T$ some terminal time, typically $T=+ \infty$, or otherwise the terminal time of an experiment, by $d_1>0$ the diffusion coefficient, by $b_1 \in \mathbb{R}^n$ a drift coefficient, and by $c_1 \in \mathbb{R}$ a constant growth/decay rate. For simplicity, we will assume that $u_0$ satisfies $\mathcal{B} u_0 = 0$ as a compatibility condition. We will consider in detail $\mathcal{B}$ to be any of the following:
\begin{enumerate}[label=\roman*.)]
    \item Dirichlet condition: $\mathcal{B} u := u$;
    \item Neumann condition: $\mathcal{B} u := \frac{\partial u}{\partial \nu}$;
    \item Robin condition: $\mathcal{B} u :=  u + \sigma \frac{\partial u}{\partial \nu}$ for real-valued constant $\sigma$.
    \item Periodic boundary conditions (toroidal geometry).
\end{enumerate}\label{list:boundary_conditions}
Here, $\partial / \partial \nu$ denotes the outward-facing unit normal vector to $\partial \Omega$. In this form, model \eqref{eq:mainPDE} has a parameter set $A:= (d, b, c) \in (0,\infty) \times \mathbb{R}^{n+1}$. 

We then have the following result, a precise execution of Steps 1-2 outlined in section \ref{sec:general_framework}.  
\begin{theorem}\label{thm:main_theorem_1}
The following hold.
\begin{enumerate}[label = \textup{(\alph*)}]
    \item Model \eqref{eq:mainPDE} is unconditionally identifiable in $\mathcal{R}$ in the sense of Definition \ref{def:identifiable_2}, where $\mathcal{R}$ is defined in \eqref{eq:identifiable_set}.
    \item Model \eqref{eq:mainPDE} is identifiable from $u_0$ in $\mathcal{A}$ if $X[A, u_0] (x,t) \not\in \ker (L[A])$ for some $t \in (0,T)$, for all $A \in \mathcal{A}$. In particular, model \eqref{eq:mainPDE} is identifiable from $u_0$ if $u_0 \not \in \ker (L[A])$ for all $A \in \mathcal{A}$. 
\end{enumerate} 
\end{theorem}
\begin{proof}
    Assume $u_0 \not\equiv 0$. We begin with part (a).
    
    Fix $A_1 \neq A_2$ in $\mathcal{R}$ and assume that there exists a nontrivial solution $u(x,t) = X[A_1, u_0](x,t) = X[A_2, u_0](x,t)$ satisfying \eqref{eq:mainPDE}. Taking the difference of these two equations and using the linearity-in-coefficients property of the operator $L$ we deduce that 
    $$
-(L[A_1]u - L[A_2]u) = - L(A_1 - A_2 ) X [ A_1, u_0 ] = - L(A_1 - A_2 ) X [ A_2, u_0 ] = 0.
    $$
    Thus, for any fixed $t_0 \in (0,T)$ (in fact, for all $t\in(0,T)$), $X[A_1, u_0] (x,t_0)$ is a nontrivial solution to $L[A]X = 0$ and so $X \in \ker (L[A])$ for $A := A_1 - A_2$. This implies that $A_1 \sim A_2$, a contradiction to their belonging to $\mathcal{R}$. This proves part (a).

    For part (b), we now fix parameter points $A_1$, $A_2$ such that $A_1 \sim A_2$ and assume that $u = X[A_1, u_0] = X[A_2, u_0]$ is a nontrivial solution to problem \eqref{eq:mainPDE} such that $u(x,t_0) \not\in \ker (L[A])$ for some $t_0 \in (0,T)$. Repeating the procedure of part (a), we find that there must hold
    $$
-L[A] u(x,t) = 0
    $$
    for all $t \in (0,T)$, and in particular, at $t_0$, a contradiction. In particular, by the continuity of the solution $X[A_i, u_0] (x,t)$ in time, we have that $\lim_{t \to 0^+} X[A_i, u_0] (x,t) = u_0(x)$ for $i=1,2$, and so $-L[A] u_0  = 0$, again a contradiction. This completes part (b).
\end{proof}
\begin{remark}
    Theorem \ref{thm:main_theorem_1} provides a sufficient but not necessary condition to guarantee identifiability of model \eqref{eq:mainPDE}. More precisely, if the initial condition belongs to the kernel, model \eqref{eq:mainPDE} may still be identifiable in $\mathcal{A}$. To see this, note that even if $A_1 \sim A_2$, it is not guaranteed that a solution exists with $\mathcal{B}u=0$ and
    $$
u_t - L[A_1] u = 0 = u_t - L[A_2] u,
    $$
    as we can only guarantee that $u_t - L[A_1] u = u_t - L[A_2] u$ from the argument above. This is why we introduce step 3 highlighted in Section \ref{sec:general_framework}.
\end{remark}

From Theorem \ref{thm:main_theorem_1}, we observe that the question of parameter identifiability for model \eqref{eq:mainPDE} can be turned into a question of identifying the elements of the kernel of the associated auxiliary elliptic operator. In general, this may not be so informative as we do not know a priori whether the kernel is finite dimensional. As it turns out, this can be ruled out for many cases of interest as there are existing results that tell us quite a bit of information about the kernel, which allow us to develop Theorem \ref{thm:main_theorem_1} further. We state the following, essentially the Fredholm alternative applied to the inverse of $L[A]$.

\begin{proposition}\label{prop:compact_operator_on_hilbert}
    Let $\Gamma \subset [0,\infty) \times \mathbb{R}^{n+1}$ denote the parameter space for which $L[A]$ is a Fredholm operator on a Hilbert space $\mathcal{H}$. Then for each fixed $A \in \Gamma$ there holds
\begin{enumerate}[label = \textup{(\alph*)}]
    \item the homogeneous problem $L[A] \phi = 0$ has only the trivial solution; otherwise
    \item the homogeneous problem $L[A] \phi = 0$ has a nontrivial solution.
\end{enumerate}
In the second case, $\dim ( \ker (L[A] ) ) < \infty$; in particular, $0$ is an eigenvalue of $L[A]$ and the dimension of the kernel of $L[A]$ corresponds to the geometric multiplicity of the zero eigenvalue.
\end{proposition}
\begin{proof}
    This is a direct consequence of the Fredholm alternative, see, e.g., \cite[Ch. 6]{brezis.2011}.
\end{proof}
\begin{remark}
    Any elliptic operator defined on an appropriate space can be extended to a Fredholm operator, e.g., when one can apply the spectral theory of compact operators \cite{kato.1980, Deimling.1985, brezis.2011}.
\end{remark}
The fact that the kernel is finite dimensional greatly reduces the possible number of initial conditions from which problem \eqref{eq:mainPDE} has indistinguishable parameter points. We can formulate Proposition \ref{prop:compact_operator_on_hilbert} in an equivalent way by isolating the lower order term $c$, which allows one to more clearly identify the set $\mathcal{A}$ that may lead to indistinguishability. 
\begin{corollary}\label{prop:compact_operator_on_hilbert_v2}
    Fix a parameter point $A \in [0,\infty) \times \mathbb{R}^{n+1}$. Then, $A \in \mathcal{A}$ if and only if $c = \lambda$, where $\lambda$ is any eigenvalue of
\begin{align}
    \begin{cases}
        -L_1[A] \phi = \lambda \phi , \quad \textup{ for } \quad x \in \Omega \\
        \mathcal{B}\phi = 0, \hspace{1.3cm} \textup{ for }\quad x \in \partial \Omega.
    \end{cases}
\end{align}
\end{corollary}
In Corollary \ref{prop:compact_operator_on_hilbert_v2}, much of the useful information is hidden in the fact that $A$ is fixed but arbitrary. We can reformulate this statement to connect more closely with eigenvalues of the operator $L$. First, we remark that the two distinct cases of Proposition \ref{prop:compact_operator_on_hilbert} allow one to identify subsets of $\mathcal{R}$ and $\mathcal{A}$. Denote
\begin{align}
     \widetilde{\mathcal{R}} := \left\{ A \in \Gamma : \text{Case 1. of Proposition \ref{prop:compact_operator_on_hilbert} holds}  \right\} \subset \mathcal{R} ; \\
     \widetilde{\mathcal{A}} := \left\{ A \in \Gamma : \text{Case 2. of Proposition \ref{prop:compact_operator_on_hilbert} holds}   \right\} \subset \mathcal{A}.
\end{align}
Ideally, these sets would agree; however, from the nature of the operator $L[A]$ and the conditions of Proposition \ref{prop:compact_operator_on_hilbert}, we cannot include the set $\{ 0 \} \times \mathbb{R}^{n+1}$ directly since the operator $L[A] = L_0(A)$ whenever $d=0$ and compactness is lost. Therefore, we must treat separately such degenerate cases. We now develop this idea fully for each of the four boundary operators introduced earlier.

\subsection{The Dirichlet boundary operator}\label{dirichlet}
In this section, we apply Theorem \ref{thm:main_theorem_1} and Propositions \ref{prop:compact_operator_on_hilbert}- \ref{prop:compact_operator_on_hilbert_v2} in one spatial dimension to identify fully the sets $\mathcal{A}$ and $\mathcal{R}$. 

To this end, set $\Omega = (0,\ell)$ for some length $\ell > 0$, and denote by $u(x,t)$ the unique solution to problem \eqref{eq:mainPDE} with nontrivial initial condition $u_0 \not\equiv 0$ with $\mathcal{B}u = u = 0$ at $x=0,\ell$. Note that if $d=0$ (i.e., if the diffusion rate is identifiable while other parameters are not), then from the Dirichlet boundary condition we find that either $b = c = 0$ or else $u_0 \equiv 0$, and so all parameters are identifiable. In particular, $\{ 0 \} \times \mathbb{R}^2 \subset \mathcal{R}$. 

\subsubsection{Case I: No drift.} 

\noindent\textbf{Steps 1-2:} For the purposes of demonstration, we first consider the simpler case
$d \frac{\partial^2}{\partial x^2} + c$ so that our parameter point is $A = (d,c) \in (0,\infty) \times \mathbb{R}$. It is well-known that the problem
\begin{align}
\begin{cases}
        - \frac{\partial^2 \phi}{\partial x^2} = \lambda \phi, \quad\quad \text{ for } \quad x \in (0,\ell) \nonumber \\
\phi = 0, \hspace{1.7cm} \text{ for } \quad x = 0,\ell
\end{cases}
\end{align}
has a discrete set of eigenvalue/eigenfunction pairs of the form
\begin{equation}\label{dirichlet_efuns}
(\phi_n , \lambda_n) = \left( \sin\left(\frac{n \pi x }{\ell} \right) , \left( \frac{n \pi }{\ell} \right)^2 \right), \quad n = 1, 2, \ldots .
\end{equation}
Therefore, by Proposition \ref{prop:compact_operator_on_hilbert_v2} we have that
$$
\mathcal{A} = \{ (d , c) \in (0,\infty) \times \mathbb{R} : c = d \lambda_n \} .
$$
By Theorem \ref{thm:main_theorem_1}, the model is unconditionally identifiable for all parameters belonging to
$$
\mathcal{R} = [0, \infty) \times \mathbb{R}^{n+1} / \mathcal{A} .
$$ 
Moreover, we have identifiability from any initial condition $u_0$ that is not a scalar multiple of an element of the kernel, namely, $\phi_n$ for some $n \in \mathbb{N}$. 

\noindent\textbf{Step 3:} With steps 1 and 2 complete, we may now identify particular solutions that lead to non-identifiability of some sets of parameters, which is to say, we identify solutions that solve both the original time-dependent problem and the auxiliary elliptic problem silmultaneously. To this end, consider pairs such that $A_1 := (d_1, c_1) \sim (d_2, c_2) = : A_2$. By construction, we have that
$$
\begin{cases}\label{eq:}
    d_1 - d_2 = d > 0; \cr
     c_1 - c_2 = d \lambda_n ,
\end{cases}
$$
for some $n \in \mathbb{N}$ fixed. By Theorem \ref{thm:main_theorem_1} we must consider initial conditions $u_0(x) \in \ker (L[A])$. In this case, we consider scalar multiples of $\phi_n$ for some $n \in \mathbb{N}$ fixed.

Since the problem is linear, we know that the solution can be written in the form $u(x,t) = \gamma(t) \phi(x)$ for some functions $\gamma$, $\phi$ via separation of variables. Moreover, since the parabolic problem enjoys uniqueness of solutions for a given initial condition $u_0$, this is the only solution, and we must have that $u_0(x) = \gamma(0) \phi_n (x)$. Through construction we have that for any $\gamma(t)$ there holds
$$
u_t - L[A_1] u = u_t - L[A_2] u ,
$$
but we do not yet know whether it is possible to also satisfy $u_t - L(A_i)u = 0$ for each $i=1,2$ as well. By substituting $u(x,t) = \gamma(t) \phi_n(x)$ into $u_t - L[A_1] u = 0$, we can solve for $\gamma(t)$ to obtain the solution
$$
u(x,t) = c_0 \mathrm{e}^{ (c_2 - d_2 \lambda_n ) t} \phi_n (x) = c_0 \mathrm{e}^{ (c_1 - d_1 \lambda_n ) t} \phi_n (x) ,
$$
for any $c_0 \in \mathbb{R}$ fixed. This solution satisfies
$$
u_t - L[A_1] u = u_t - L[A_2] u  = 0,
$$
subject to the initial condition $u(x,0) = c_0 \phi_n (x)$, and $u = 0$ at $x=0,\ell$. 

\noindent\textbf{Step 4:} Under the further assumption that the measured quantity is nonnegative, we remove all eigenfucntions for $n \geq 2$ and retain only the first $n=1$ which leads to indistinguishable parameter points corresponding to a positive solution.

In Figure \ref{fig3}, we display the set of lines in the parameter space $\mathbb{R}^2$ coming from $\mathcal{A}$. From Step 3, we obtain the red and blue lines; with the addition of Step 4, we remove all blue lines and retain only the red line (corresponding to the first eigenfunction, the only one which remains positive everywhere).

From this example, we see that the temporal component interacts with the spatial component in a rather pathological way, allowing one to obtain indistinguishable points, leading to non-identifiability of model \eqref{eq:mainPDE}. This is a scenario for which belonging to the set $\mathcal{A}$ (in the sense of the equivalence relation) yields non-identifiable solutions explicitly.

It is this model for which we demonstrate structural non-identifiability in \cref{fig1}. We consider a unit domain with $\ell = 1$ and an initial condition $u_0(x,0) = \phi_1(x) = \sin(\pi x)$. Taking $A_1 = (c_1,d_1) = (1,0.05)$, our methodology indicates that all other parameter combinations  $A_2 = (c_2,d_2)$ will be indistinguishable provided that $c_1 - c_2 = (d_1 - d_2)\pi^2 \Leftrightarrow c_1 - d_1 \pi^2 = c_2 - d_2 \pi^2$. In other words, the combination $c - d\pi^2$ is structurally identifiable, while the constituent parameters $c$ and $d$ are not.

		
		

 \begin{figure}
     \centering
     \begin{subfigure}[t]{.5\textwidth}
     \centering\captionsetup{width=.8\linewidth}
     \includegraphics[width=0.9\linewidth]{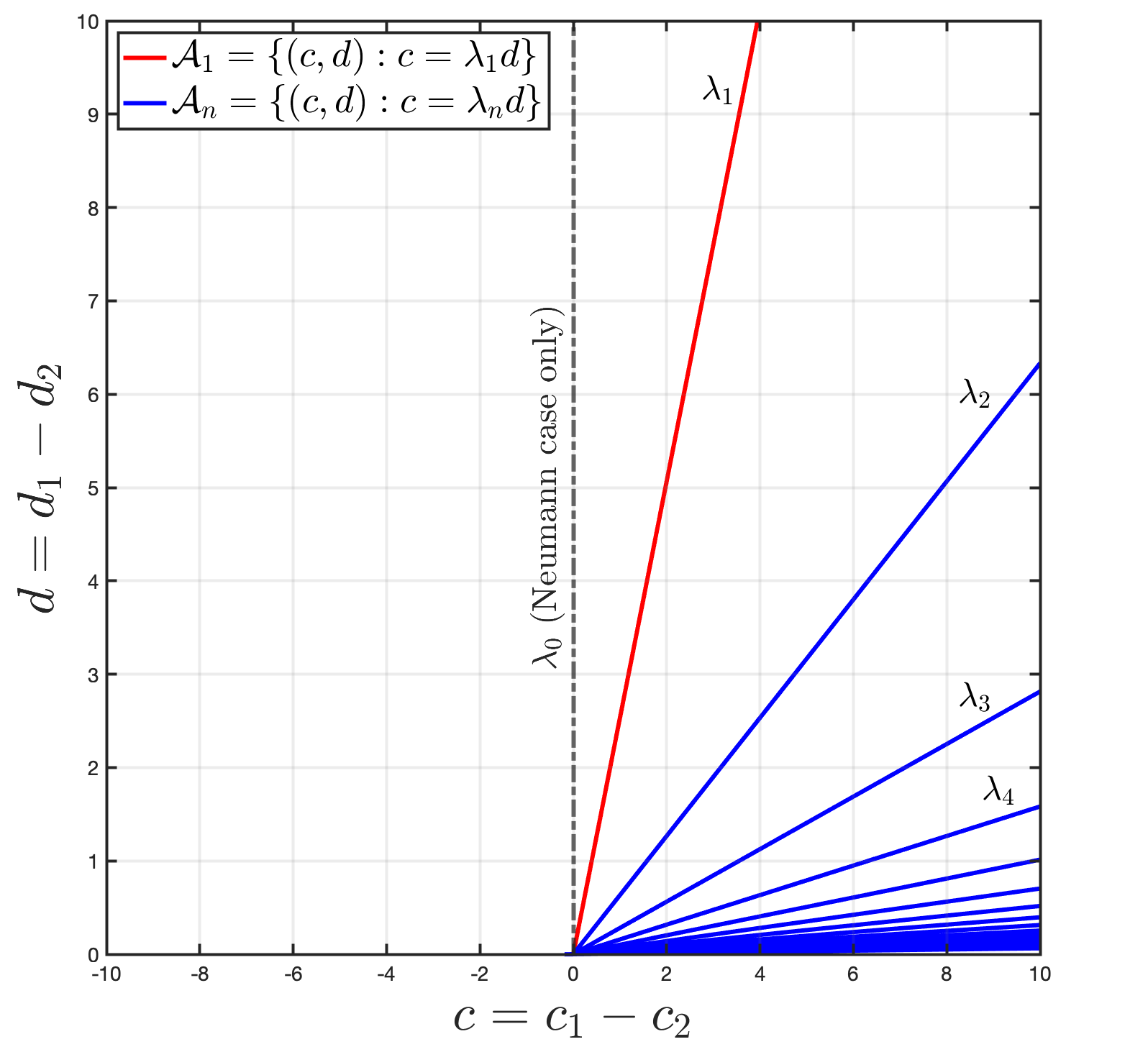}
     \caption{For the Dirichlet case, $\mathcal{A}$ is given by the red line and blue lines. The red line corresponds to the only eigenfunction which is positive everywhere in the domain. In the Neumann case, $\mathcal{A}$ is also comprised of countably many lines corresponding to each eigenvalue; however, the positive eigenfunction now corresponds to $\lambda_0 = 0$.}
     \label{fig3}
     \end{subfigure}%
     \begin{subfigure}[t]{.5\textwidth}
     \centering\captionsetup{width=.8\linewidth}
     \includegraphics[width=0.9\linewidth]{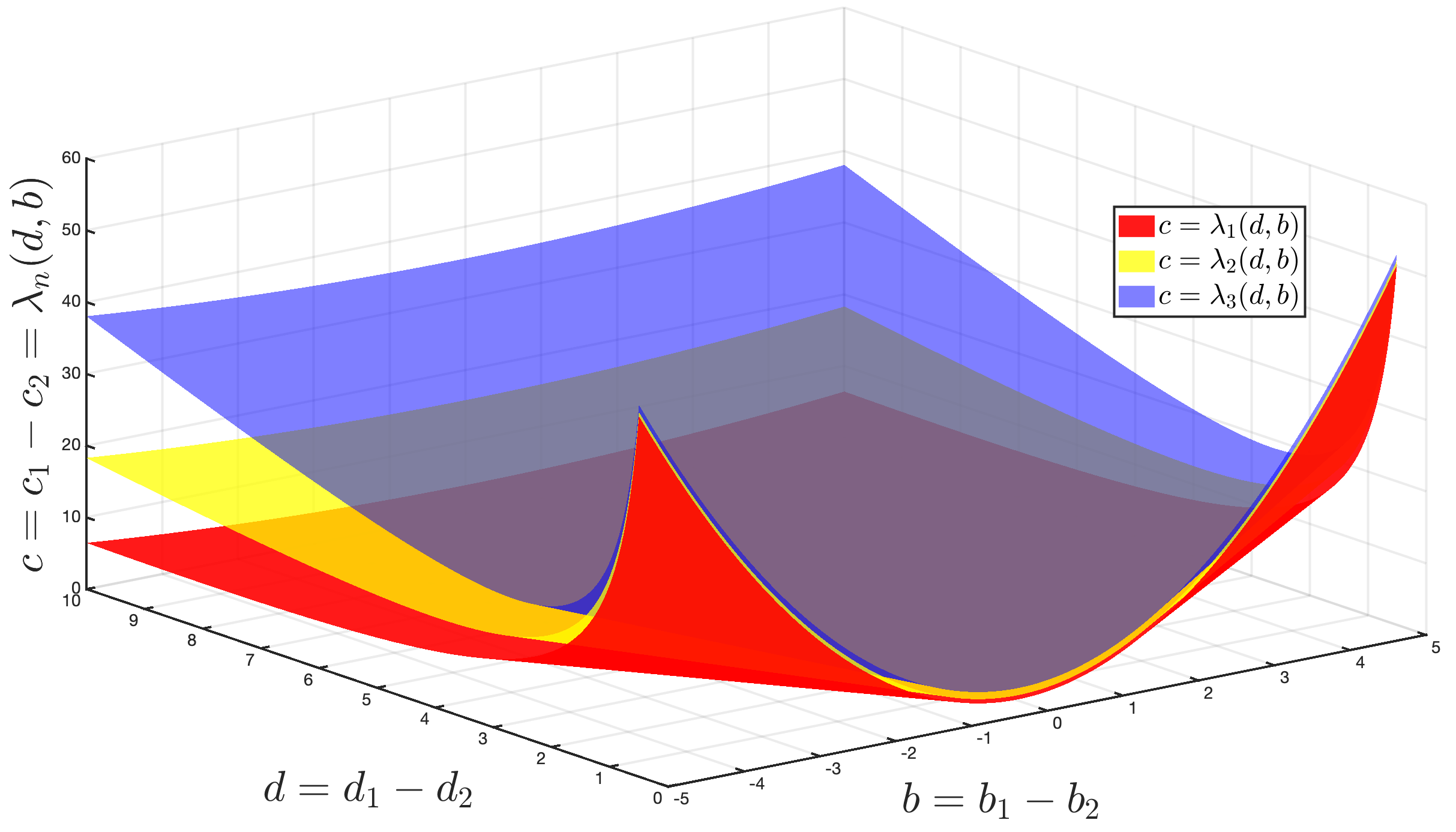}
     \caption{A depiction of the set $\mathcal{A}$ for the first three eigenvalues. The surfaces are given directly by $c$ as a function of $(d,b)$ via $c = \lambda_n (d,b)$. The red surface corresponds to the positive eigenfunction. When we take $b=0$, the cross-section recovers the red line in panel \textbf{(a)}. }
     \label{fig3_b}
     \end{subfigure}
     \label{fig3_comb}\caption{A depiction of the set $\mathcal{A}$ for the Dirichlet problem without drift (a) and with drift (b).}
 \end{figure}

\subsubsection{Case II: Constant drift.} 

\noindent\textbf{Steps 1-2:} For the general case $L[A] = d \partial^2 / \partial x^2 + b \partial / \partial x + c$, the relation between the coefficients and eigenvalues of the operator is more complicated. However, it is still possible to describe a large set of identifiable parameters, as well as construct non-identifiable solutions. First, we note that for any $d>0$ the problem
\begin{align}
\begin{cases}
        - d \frac{\partial^2 \phi}{\partial x^2} - b \frac{\partial \phi}{\partial x} = \lambda \phi, \quad\quad \text{ for } \quad x \in (0,\ell) \nonumber \\
\phi = 0, \hspace{3cm} \text{ for } \quad x = 0,\ell
\end{cases}
\end{align}
also enjoys a discrete set of eigenvalue/eigenfunction pairs of the form
$$
( \phi_n , \lambda_n (d,b) ) = \left(\mathrm{e}^{-b x/2d } \sin(\sqrt{\lambda_n} x), \frac{b^2}{4d} + d \frac{n^2 \pi^2}{\ell^2}  \right),
$$
from which one obtains the no-drift case by taking $b=0$. Unlike the $b=0$ case, however, we cannot describe the eigenvalue/eigenfunction pairs independent from the parameters $(d,b)$, i.e., $\lambda_n = \lambda(d,b)$ is the object we have to work with. This is different from the no drift case where we were able to abuse the homogeneity of $\lambda_n$, i.e., $\lambda_n(d) = d \lambda_n(1)$. 

\noindent\textbf{Step 3:} As in the no-drift case, by Proposition \ref{prop:compact_operator_on_hilbert_v2}, $L[A]$ has a nontrivial kernel if and only if $c = \lambda_n(d,b)$ for some $n$, and $\mathcal{A}$ is as in the no-drift case with new eigenvalues $\lambda_n(d,b)$:
$$
\mathcal{A} = \{ (d , b,  c) \in (0,\infty) \times \mathbb{R}^2 : c = \lambda_n (d,b) \} .
$$
Consider now fixed parameter points $A_1 \sim A_2$. This yields a system of equations:
\begin{align}
    \begin{cases}
        d_1- d_2 = d > 0 , \cr
        b_1 - b_2 = b \in \mathbb{R} ,  \cr
        c_1 - c_2 = \lambda_n (d,b) .
    \end{cases}
\end{align}
As in the no-drift case, separation of variables yields a solution of the form $u(x,t) = \gamma(t) \phi_n (x)$ satisfying $u_t - L[A_1] u = u_t - L[A_2] u$. Substituting this into $u_t - L[A_1] u = 0$, after some computation, yields the necessary relation
$$
0 = \left( \frac{\gamma^\prime (t)}{\gamma(t)} + \frac{d_1 \lambda_n (d,b) }{d} + c_1  \right) u + \left( b_1 d - d_1 b  \right) \frac{\gamma(t) (\phi_n)_x}{d}.
$$
From the second term on the right hand side, we observe a sufficient relation between the diffusion and drift to remove it from the equation:
$$
\frac{d_1}{d_2} = \frac{b_1}{b_2} ,
$$
which, in particular, indicates that for non-identifiable solutions to exist, the sign of $b_1$ and $b_2$ must be the same.  Using the remaining relation, we may then solve for $\gamma(t)$  as a first order, linear ordinary differential equation  to obtain the non-identifiable solution
$$
u(x,t) = c_0 \mathrm{e}^{ ( c_1 - d_1 \lambda_n (d_1 - d_2, b_1 - b_2) )t} \phi_n (x),
$$
which satisfies $u_t - L[A_1] u = u_t - L[A_2] u = 0$ and $u = 0$ at $x=0,\ell$. 

Comparing with the no-drift case, we again find that the spatial and temporal components must interact in a pathological way to lead to non-indentifiability; note, however, that in the second case we identify an additional constraint relating the ratio of diffusion coefficients with the ratio of drift coefficients. This is reasonably intuitive: indistinguishable parameters cannot be produced from drift pointing in opposite directions.

\noindent\textbf{Step 4:} We may now execute an optional Step 4. For example, suppose that our solution (i.e., our measurement) $u(x,t)$ is nonnegative. Then, we can safely remove \textit{all} non-identifiable solutions \textit{except for the first one}, which is to say, the non-identifiable solution corresponding to the first eigenfunction $\phi_1(x)$ (see the red surface of Figure \ref{fig3_b}). Indeed, it is the only eigenfunction which can be chosen nonnegative everywhere in the domain, and all other eigenfunctions change sign and therefore have an incompatible negative region.

\subsection{The Neumann, Robin \& periodic boundary conditions}\label{otherbc}

From the analysis of the Dirichlet problem of the previous section, we briefly discuss some other common boundary conditions. In each case, the set of parameters that may possibly lead to non-identifiability relates to the spectrum of the related elliptic eigenvalue problem.

\subsubsection{The Neumann condition.} 

First, we again note that if $d=0$ (i.e., if $d_1 = d_2$), then our auxiliary equation is of the form $- b \phi_x - c \phi = 0$ and solution are of exponential form. From the homogeneous Neumann condition, the only possibility is that $\phi$ is constant everywhere, and there must hold $c = c_1 - c_2 = 0$ while any $b \in \mathbb{R}$ is valid. We can understand this further as follows.

Intuitively, we anticipate that there is an issue with parameter distinguishability under a homogeneous Neumann boundary condition because any solution to the ordinary differential equation
$$
y^\prime(t) = c y
$$
will satisfy model \eqref{eq:mainPDE} when the initial condition is constant over space. From the viewpoint of Theorem \ref{thm:main_theorem_1} and Propositions \ref{prop:compact_operator_on_hilbert}-\ref{prop:compact_operator_on_hilbert_v2}, we can reach the same conclusion as follows. Consider the eigenvalue problem
\begin{align}
\begin{cases}
        - d \frac{\partial^2 \phi}{\partial x^2} - b \frac{\partial \phi}{\partial x} = \lambda \phi, \quad\quad \text{ for } \quad x \in (0,\ell) \nonumber \\
\frac{\partial \phi}{ \partial x} = 0, \hspace{2.75cm} \text{ for } \quad x = 0,\ell
\end{cases}
\end{align}
for $d>0$, $b \in \mathbb{R}$ fixed. One can again solve for the eigenvalue/eigenfunction pairs $(\phi_n , \lambda_n)$ for $n = 1, 2, \ldots$. They are again discrete and each lead to a curve of non-identifiable solutions in a particular parameter space exactly as in the Dirichlet case. However, there is also $\lambda_0 = 0$ as a legitimate eigenvalue with constant (nontrivial) eigenfunction. Then, since we obtain a non-identifiable solution from scalar multiples of members of the kernel, the set
$$
 \{ (d,b,c) \in [0,\infty) \times \mathbb{R}^2 : c = \lambda_0 = 0 \}
$$
is precisely $[0,\infty) \times \mathbb{R} \times \{ 0 \}$. In other words, for initial conditions $u_0$ spatially constant, the diffusion and drift coefficients $(d,b)$ are not identifiable \textit{anywhere} in $[0,\infty) \times \mathbb{R}$. Together, we identify the set
$$
\mathcal{A} = \{ (d,b,c) \in (0,\infty) \times \mathbb{R}^2 : c = \lambda_n (d,b),\, n \geq 1 \} \cup \{ (d,b,c) \in [0,\infty) \times \mathbb{R} \times \{ 0 \} \}.
$$

\subsubsection{The Robin condition.} Different from the Dirichlet and Neumann cases, the Robin condition allows for the possibility of non-identifiability when $d=0$. Indeed, an eigenvalue/eigenfunction pair $(\phi_0, \lambda_0 (d,\sigma) = (\mathrm{e}^{-x/\sigma}, b/\sigma)$ exists for any constant $\sigma \in \mathbb{R}$ coming from the boundary condition. We may then consider fixed parameter points $A_1 \sim A_2$ satisfying
\begin{align}
    \begin{cases}
        d_1- d_2 = 0 , \cr
        b_1 - b_2 = b \in \mathbb{R} ,  \cr
        c_1 - c_2 = \lambda_0 (d,b) = b/\sigma = (b_1 - b_2)/\sigma .
    \end{cases}
\end{align}
We may then substitute an ansatz of the form $u(x,t) = \gamma(t) \phi_0 (x)$ into $u_t - L_0 [A_1] u = 0$ to obtain a solution of the form
$$
u(x,t) = \mathrm{e}^{(d_1 / \sigma^2 - b_1 / \sigma + c_1)t} \phi_0 (x),
$$
which also solves $u_t - L_0 [A_2] u = 0$, and so $A_1$, $A_2$ are indistinguishable. 

When $d>0$, it is well-known that the Robin condition produces a spectrum between the Dirichlet and Neumann conditions (i.e., when $\sigma \to 0$, we recover the Dirichlet condition; when $\sigma \to \pm \infty$ we recover the Neumann condition). In particular, the associated eigenvalue problem has a discrete spectrum with $0 < \lambda_1 \leq \lambda_2 \leq \cdots$, and we can produce non-identifiable solutions exactly as in the Dirichlet or Neumann cases. Notice that for any $\sigma \neq 0$, $0$ is no longer an eigenvalue and so constant functions become initial conditions from which model \eqref{eq:mainPDE} is identifiable everywhere. 

Together, we identify the set 
$$
\mathcal{A} = \{ (d , b,  c) \in (0,\infty) \times \mathbb{R}^2 : c = \lambda_n (d,b) \} \cup \{ (d,b,c) \in \{ 0 \} \times \mathbb{R}^2 : c = b/\sigma \} .
$$

\subsubsection{The periodic condition.} It is also interesting to consider briefly the periodic boundary condition. For simplicity, we consider the no-drift case and we consider the domain to be a torus $\mathbb{T} := \mathbb{R}/ 2 \pi \mathbb{Z}$. Then the eigenvalues are simply $\lambda_n = n^2$ for $n \geq 0$; however, for $n\geq1$ the geometric multiplicity of each eigenvalue is $2$, given by the eigenfunctions $\phi_{n,1} = \cos (n x)$ and $\phi_{n,2} = \sin (nx)$. In the previous cases, we only obtain a non-identifiable solution from a single eigenfunction. This is because each eigenvalue was simple and therefore enjoys a unique eigenfunction (up to scalar multiplication). For the periodic boundary condition, each eigenvalue (aside from $\lambda_0$) has geometric multiplicity $2$, and therefore non-identifiable solutions are generally given by linear combinations of these eigenfunctions:
$$
u(x,t) = \gamma_1 (t) \phi_{n,1} (x) + \gamma_2 (t) \phi_{n,2} (x) .
$$
The same analysis applies for the case with drift, resulting in a transcendental equation for the eigenvalues, all of which are positive, depending on $b,d$. 

This case, while less biologically reasonable, highlights the importance of the geometric multiplicity of the eigenvalue. In the Dirichlet, Neumann, or Robin cases, a geometric multiplicity of one for all eigenvalues ensures that an experimental treatment with two distinct initial data (not scalar multiples) eliminate non-identifiability; for the periodic boundary condition, one requires three distinct initial data such that they do not both belong to the span of the associated eigenfunctions.

\section{Nonlinear reaction-diffusion equations and systems}\label{sec:inhomogeneous_pde}

\subsection{Nonlinear reaction-diffusion equations}

Using the approach developed in previous sections, we can treat some nonlinear equations which are linear in their parameters. For simplicity, we treat problems with diffusion and reaction only:
\begin{align}\label{eq:nonlinear_PDE}
    \begin{cases}
        u_t - d_1 \Delta u = f(x,u; B_1), \quad\quad \text{ for } \quad (x,t) \in \Omega \times (0,T), \cr 
        \mathcal{B}u = 0, \hspace{3.525cm} \text{ for } \quad (x,t) \in \partial \Omega \times (0,T) , \cr
        u(x,0) = u_0(x), \hspace{2.2cm} \text{ for } \quad x \in \overline{\Omega},
    \end{cases}
\end{align}
where $d_1 > 0$ and $f$ is a \textit{nonlinear} function depending linearly on a parameter point $B_1 \in \mathbb{R}^m$ for some $m \geq 1$. The dimensionality of the parameter $B_1$ is arbitrary in principle, but larger values of $m$ increase the difficulty of applying this approach. Any polynomial of degree two or greater will fall into this class. One such example is a logistic-growth functional response of the form
$$
 f(x,u;B_1) =  f(u,B_1) = a_1 u - b_1 u^2,
$$
where $B_1 = (a_1,b_1) \in \mathbb{R}^2$ so that $m=2$.

We identify the following for auxiliary parameter points $d$ and $B$.
\begin{align}\label{identifiable_nonlinear_R}
    \mathcal{R} = \{ (d, B) \in [0, \infty) \times \mathbb{R}^m : - d \Delta \psi = f(x, \psi; B) \text{ has a unique solution}. \}
\end{align}
We then identify the remaining possible parameter sets which may lead to non-identifiability.
\begin{align}\label{identifiable_nonlinear_A}
\mathcal{A}_{\textup{I}} =& \{ (d, B) \in [0,\infty) \times \mathbb{R}^m : - d \Delta \psi = f(x, \psi; B) \text{ has a discrete spectrum of solutions}. \} \nonumber \\
    \mathcal{A}_{\textup{NI}} = &\{ (d, B) \in [0,\infty) \times \mathbb{R}^m : - d \Delta \psi = f(x, \psi; B) \text{ has a continuous spectrum of solutions}. \}
\end{align}

We prove the following, which establishes unconditional identifiability if solutions to the auxiliary elliptic problem are discrete.
\begin{theorem}\label{thm:nonlinear_scalar}
    Suppose $u$ is a smooth solution (e.g., $C^{2,1} (\Omega\times (0,T))$, but once differentiable in space and time is sufficient) and that $u_t \not\equiv 0$. Then model \eqref{eq:nonlinear_PDE} is unconditionally identifiable in $\mathcal{R} \cup \mathcal{A}_{\textup{I}}$ as defined in \eqref{identifiable_nonlinear_R}-\eqref{identifiable_nonlinear_A}. In particular, if $\mathcal{A}_{\textup{NI}} = \emptyset$, then model \eqref{eq:nonlinear_PDE} is globally unconditionally identifiable.
\end{theorem}
\begin{proof}
    Fix $u_0 \not\equiv 0$. As in the homogeneous cases, we assume to the contrary that there exists a solution $u = X[d_1, B_1, u_0] (x,t) = X[d_2, B_2, u_0] (x,t)$ for all $(x,t) \in \Omega \times (0,T)$, where the parameter points are distinct, i.e., $(d_1,B_1) \not\equiv (d_2,B_2)$. Without loss of generality, we assume $d_1 \geq d_2$. As in the homogeneous case, we take the difference of these two equations to obtain the auxiliary elliptic equation
\begin{align}\label{eq:nonlinear_aux}
    -(d_1 - d_2) \Delta u = f(x,u; B_1 - B_2), \quad \mathcal{B}u = 0 .
\end{align}
    Setting $d:= d_1 - d_2 \geq 0$, $B := B_1 - B_2$, we obtain that for any $(d,B) \in \mathcal{R}$ there necessarily holds $u(x,t) = \psi(x)$ for all $t \in (0,T)$. Different from the homogeneous case, we no longer have uniqueness \textit{up to scalar multiples}, and so we conclude that $u_t = \frac{\textup{d}}{\textup{d}t} \psi = 0$, a contradiction.

    Similarly, if $(d,B) \in \mathcal{A}_{\textup{I}}$ we find that $u(x,t) = \psi(x)$ for some $\psi$ satisfying \eqref{eq:nonlinear_aux}; since the solutions corresponding to parameter points belonging to $\mathcal{A}_{\textup{I}}$ are discrete, we conclude that $u(x,t) = \psi(x)$ for some fixed $\psi$ for all $t \in (0,T)$, and we again reach a contradiction to $u_t \not\equiv 0$.
\end{proof}
\begin{remark}
    In the proof above, we could equivalently assume that $u_0 (x) \not\equiv \psi$, where $\psi$ is a solution to the auxiliary elliptic problem \eqref{eq:nonlinear_aux}. Indeed, from the assumed continuity of the solution $u(x,t)$, we would conclude that $\lim_{t  \to 0^+} u(x,t) = \psi(x)$. A key difference between a nonlinear case and the homogeneous case is that we lose the  degree of  freedom to use scalar multiples of the solution to the elliptic problem to construct a valid solution solving the time-dependent problem. 
\end{remark}

\begin{example}\label{ex:constant_logistic}
    Consider the logistic form $f(x,u;B_1 ) = a_1 u - b_1 u^2$ introduced above. $u(x,t) \equiv 0$ is always a solution for the boundary conditions considered here. We assume that the solution $u(x,t)$ is nontrivial and non-constant in time. Set $d:= d_1 - d_2 \geq 0$, $B := (a,b) = (a_1 - a_2, b_1 - b_2)$ as in \eqref{eq:nonlinear_aux}. 

    When $d = 0$, one necessarily finds that $u \equiv a/b$, a contradiction to the assumption $u_t \not\equiv 0$.
    
    For any $d>0$, it is well-known that for $a \in \mathbb{R}$, $b > 0$, any nontrivial solution is unique, see, e.g., \cite{Pao.1992}. This follows from the concavity of the functional response (a subhomogeneity condition is sufficient; see \cite[Ch. 2]{Zhao.2017}). Therefore, any parameter combinations such that $a_1 - a_2 \in \mathbb{R}$, $b_1 - b_2 > 0$ lead to distinguishability and therefore identifiability by Theorem \ref{thm:nonlinear_scalar}.
    
For cases where $b_1 - b_2 < 0$, we appeal to symmetry of the problem. Indeed, we notice that if $w$ solves the auxiliary elliptic problem with $b<0$, then $z := - w$ satisfies the same equation with $b>0$. Consequently, if there exist multiple nontrivial solutions for parameters $(a,b) \in \mathbb{R} \times \mathbb{R}^-$, we would identify multiple solutions for parameters belonging to $(a,b) \in \mathbb{R} \times \mathbb{R}^+ \subset \mathcal{R}$, a contradiction to the uniqueness of solutions to the auxiliary elliptic problem highlighted above. In particular, $\mathcal{A}_{\textup{NI}} = \emptyset$.

Therefore, by Theorem \ref{thm:nonlinear_scalar}, model \eqref{eq:mainPDE} with the logistic growth function subject to any of the homogeneous Dirichlet, Neumann, Robin, or periodic boundary conditions is globally unconditionally identifiable in the sense of Definition \ref{def:identifiable_1} so long as there is some temporal variation in the observation, which is to say, $u_t \not\equiv 0$.
\end{example}

Sometimes we also consider spatially-dependent parameters, see for example \cite{Baym.2016, Jacobs.2024}, motivating the following example.

\begin{example}\label{ex:hetero_logistic}
    An identical analysis to that presented above applies to spatially dependent parameters. Consider $f(x,u; B_1) = m_1(x) u - b_1 u^2$ and take $\mathcal{B}$ to be the Neumann boundary operator. We now have a spatially dependent parameter $m_1(x)$, and a constant parameter $b_1 > 0$ as in Example \ref{ex:constant_logistic}. Taking the difference between two equations potentially satisfied by a non-identfiable solution $u$, we obtain the elliptic equation $-d \Delta \psi = m(x) \psi - b \psi^2$, where $m (x) := m_1(x) - m_2(x)$ and $b := b_1 - b_2$. Assuming that $b>0$ and that $m \not\equiv \textup{const.}$, any positive solution $\psi$ is unique, see \cite[Ch. 4]{Ni.2011}. If $m \equiv \textup{const}$, then $\psi \equiv b/m$ whenever $m \not\equiv 0$, otherwise $\psi \equiv 0$ is the only solution. Symmetry yields the same conclusion whenever $b<0$. In any case, we find that model \eqref{eq:nonlinear_PDE} subject to a homogeneous Neumann boundary condition with heterogeneous logistic growth $f(x,u) = m_1 (x) u - b_1 u^2$ is globally unconditionally identifiable so long as $u_t \not\equiv 0$. 
\end{example}

\subsection{Systems of reaction-diffusion equations}
We conclude with an example of a non-identifiable solution for a coupled system of reaction-diffusion equations. We do not seek to be exhaustive in this case, but instead seek to demonstrate how non-identifiability can arise. Consider the following model describing cell motility with linear diffusion \cite{Vittadello:2018}:
\begin{align}\label{eq:cell_motility_system}
\begin{cases}
        u_t - d_u \Delta u = a_{11} u - a_{12} v, \quad \text{ for } \quad (x,t) \in \Omega \times (0,T), \cr
        v_t - d_v \Delta v = - a_{21} u + a_{22} v , \quad \, \text{ for } \quad (x,t) \in \Omega \times (0,T), \cr
        \mathcal{B}u = \mathcal{B} v = 0 , \hspace{2.9cm} \text{ for } \quad (x,t) \in \partial \Omega \times (0,T),
\end{cases}
\end{align}
where $a_{ij} > 0$ are constants, $i,j=1,2$, and $\mathcal{B}$ is any of the Dirichlet, Neumann, or Robin boundary operators. We may then construct a non-identifiable solution using the theory from Section \ref{sec:homogeneous_pde} as follows. Let $u = u_i$ solve $u_t - d_i \Delta u - c_i u = 0$, $i=1,2$, where $(d_i, c_i)$ are chosen such that $u_1 \equiv u_2$ while $(d_1, c_2) \not\equiv (d_2,c_2)$. By the results of Section \ref{eq:cell_motility_system}, we know such a solution exists for any boundary operator considered here. By direct substitution, we then find that $(u,v) := ( \kappa _u u_1, \kappa_v u_1)$ with any positive constants $\kappa_u$, $\kappa_v$ will solve \eqref{eq:cell_motility_system} with
\begin{align}
\begin{cases}
d_u := d_1  ;\quad a_{11} :=  c_1 + \delta_1 ; \quad a_{12} := \delta_1 \frac{\kappa_u}{\kappa_v}, \nonumber \\
 d_v := d_2  ;\quad a_{22} :=  c_2 + \delta_2 ; \quad a_{21} := \delta_2 \frac{\kappa_v}{\kappa_u},
    \end{cases}
\end{align}
where $\delta_i>0$ are arbitrary constants. We remind readers that the restriction on $d_i,c_i$ tell us that each $c_i$ can be written in terms of $d_i$ and the eigenvalues of $-\Delta$ depending on the boundary condition chosen.

As $\delta_i$ are arbitrary but yield a consistent set of
 solutions to problem \eqref{eq:cell_motility_system}, we have obtained a continuous spectrum of non-identifiable solutions $(u,v):=(u_{(\delta_1,\delta_2)}, v_{(\delta_1,\delta_2)})$ parametrised by $(\delta_1, \delta_2)$.

\section{Ramifications for practical identifiability}\label{sec:practical}

Results in the preceding sections show that, for scalar parabolic PDEs that are fully observed, non-identifiability is in some sense pathological: models are structurally identifiable from initial conditions that do not take a very specific form. For example, results in \Cref{dirichlet} demonstrate that the linear reaction-diffusion equation with Dirichlet boundary conditions on a unit domain is structurally non-identifiable if and only if the initial condition is of the form of the dominant eigenfunction $u_0(x) = \sin(\pi x)$. Any perturbation to this initial condition will render all model parameters structurally identifiable. However, for linear models, at least, this perturbation will decay as the solution tends toward a multiple of the dominant eigenfunction. Therefore, while model parameters are structurally identifiability for almost all initial conditions, we expect to encounter practical non-identifiability for data collected from experiments initialised near the dominant eigenfunction.

\begin{figure}[!t]
	\centering
	\includegraphics{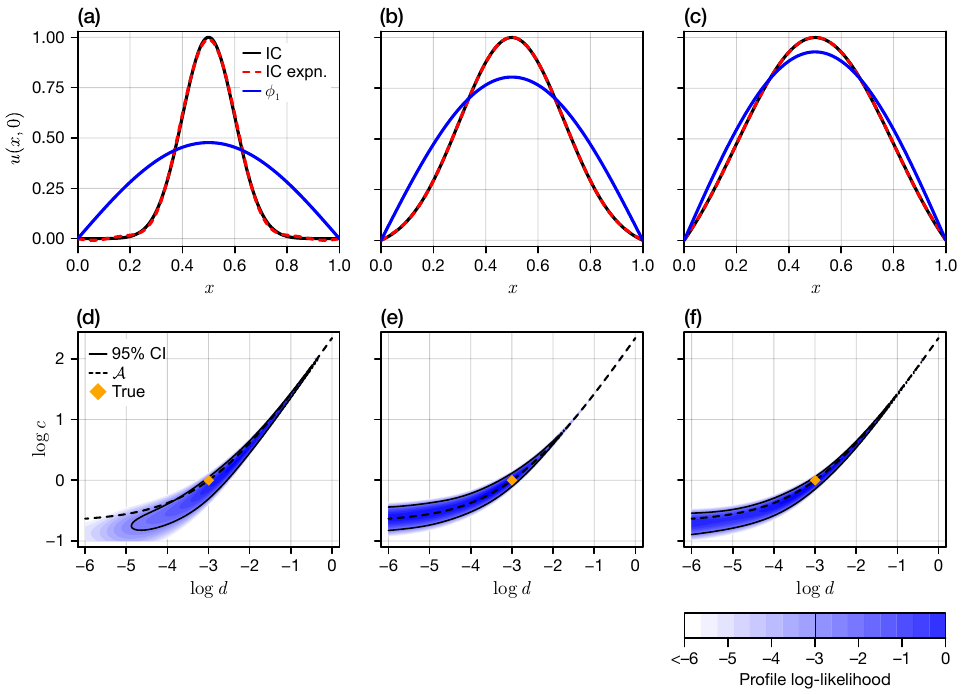}
	\caption[Figure 4]{Practical identifiability of parameters in the linear reaction-diffusion equation subject to Dirichlet boundary conditions. (a--c) Initial conditions used to produce synthetic data. Initial conditions are based on a Gaussian centred at $x = 1/2$ with standard deviations $\omega = 0.1,0.2,0.3$, translated and normalized such that $u_0(1/2) = 1$ and that the boundary conditions are satisfied (black solid). To simplify the analysis, the initial conditions used are constructed from a truncated eigenfunction expansion of each Gaussian (red dashed). The leading order term is shown in blue. (d--f) Bivariate profile log-likelihoods for each initial condition, normalized to a maximum of zero. A 95\% confidence region is calculated based on a normalized log-likelihood threshold of approximately 2.997, arising from the likelihood-ratio test. We also show the true value (orange diamond) and the theoretically indistinguishable set for an initial condition exactly matching the dominant eigenfunction (black dashed).}
	\label{fig4}
\end{figure}

To demonstrate the relationship between initial condition and practical identifiability, we consider a likelihood-based inference to recover parameters $A_1 = (c,d) = (1,0.05)$ in \cref{eq:mainPDE} subject to Dirichlet boundary conditions on a unit domain with $b = 0$. We additionally assume that the initial condition is to be inferred, and is of the form
	\begin{equation}\label{practical_ic}
		u(x,0) = \sum_{n=1}^N C_n \, \phi_n(x),
	\end{equation}
where $\phi_n(x)$ are the eigenfunctions (\cref{dirichlet_efuns}) and $C_n$ the (unknown) coefficients in the partial eigenexpansion for the initial condition truncated at $N = 8$. This choice of initial condition also allows us to avoid issues relating to a numerical discretisation with a diffusion coefficient that  potentially  varies across several orders of magnitude, as the coupled problem given through \cref{eq:mainPDE,practical_ic} has an analytical solution of the form
	\begin{equation}
		u(x,t) = \sum_{n=1}^N C_n \mathrm{e}^{\lambda_n t} \phi_n(x),
	\end{equation}
with the corresponding eigenvalues given by \cref{dirichlet_efuns}. We choose $u(x,0)$ (and hence the parameters $\{C_n\}_{n=0}^N$) based on the truncated eigenfunction expansion of a Gaussian density function centered at $x = 1/2$ with standard deviation $\omega$ and both scaled and translated such that $u(x,0) = 1$ and that the boundary conditions are satisfied (\cref{fig4}a--c). Varying $\omega \in \{0.1,0.2,0.3\}$ allows us to create initial conditions that are very similar to the leading eigenfunction (\cref{fig4}c) and that are very different (\cref{fig4}a). 

We consider a spatiotemporal dataset collected at $x = x_i \in \{0,0.1,\cdots,1\}$ and $t  = t_j \in \{0,0.1,\cdots,2\}$. We further assume that observations, denoted by $$y(x_i,t_j) = u(x_i,t_j) + \varepsilon(x_i,t_j),$$ are subject to additive Gaussian noise with covariance function
	\begin{equation}
		\mathrm{cov}\big(\varepsilon(x_{i_1},t_{j_1}), \varepsilon(x_{i_2},t_{j_2})\big) = \left\{\begin{array}{ll}
				\sigma^2 \mathrm{e}^{-\eta |x_{i_1} - x_{i_2}|} & t_{j_1} = t_{j_2},\\
				0 & t_{j_1} \neq t_{j_2},
		\end{array}\right.
	\end{equation}
where $\eta$ and $\sigma^2$ (both to be inferred) determine the strength of the spatial correlation and the overall noise variance, respectively, and are set to the known parameters $(\eta,\sigma) = (10,0.3)$.

For synthetic datasets generated using each of the initial conditions in \cref{fig4}a--c, we perform a two-dimensional profile likelihood analysis to establish the practical identifiability of the model parameters $(c,d)$ \cite{Pawitan.2013}. This analysis involves maximizing the likelihood over the remaining parameters (in our case, the coefficients relating to the initial condition $\{C_n\}_{n=0}^N$) while the parameters of interest remain fixed. An approximate 95\% confidence region can then be constructed based upon a comparison to the maximum likelihood estimate through a likelihood ratio test. Results in \cref{fig4}d--f show that, for all initial conditions, the 95\% confidence region follows a shape predicted by the indistinguishable set $\mathcal{A}$, corresponding to a constant dominant eigenvalue. In particular, parameters inferred from initial conditions constructed with $\omega = 0.2$ and $0.3$ are \textit{practically non-identifiable}: they cannot be identified to within several orders of magnitude. In the case that $\omega = 0.1$, we can establish a lower bound and we observe in \cref{fig4} that the 95\% confidence region is contained entirely within the domain shown. In all cases, we observe a confidence region that narrows significantly around $\mathcal{A}$ for large $d$. We highlight that this does not necessarily imply that large $d$ regimes are distinguishable, but rather that the model becomes increasingly sensitive to changes in $c$ as $d \rightarrow \infty$.

\section{Discussion}\label{sec:discussion}

Despite the ever-increasing prevalence of spatiotemporal data, tools for the structural identifiability analysis of the necessitated PDE models remain in their infancy. Recent work extends the well-established differential algebra approach to both fully and partially observed PDE models, enabling, for the first time, a very general assessment of identifiability in situations where the state-space is, in some sense, adequately explored. However, many considerations in the application of PDEs differ fundamentally from that of ODE models: potentially most pertinent of which are limitations relating to boundary conditions and spatially inhomogeneous initial conditions arising from experimental or practical constraints. In this work, we present a new perspective by studying initial-condition-dependent structural identifiability in a class of fully observed parabolic equations that are linear in their parameters. Our direct approach is based upon studying the existence and uniqueness of solutions of an auxiliary elliptic equation, for which we draw on a rich body of theory that had remained disconnected from the question of structural identifiability.

For scalar linear parabolic models, we demonstrate that structural non-identifiability is in some sense a pathological consequence of an initial condition that belongs to the kernel of the elliptic operator. Therefore, these models always retain so-called multi-experiment identifiability provided that at least one initial condition that is not a scalar multiple of another is considered. Furthermore, since the set of  indistinguishable parameter points, $\mathcal{A}$, is empty in cases of identifiability, we can also say that the scalar linear parabolic models are \textit{globally identifiable}. Indeed, the differential algebra approach, through the coefficients of a monic polynomial in the state variable and its derivatives, will indicate that these models are globally identifiable. Underlying this conclusion is that the constituent monomials are linearly independent. This can be held, in general, valid for PDE models; however, it is clear from the auxiliary elliptic equation that the state variable and its spatial derivatives will be linearly \textit{dependent} in the constrained state space arising from a non-identifiable initial condition. Consequentially, any perturbation to this initial condition (including through a spatially inhomogeneous input) will render the model identifiable. However, as we investigate in \cref{sec:practical}, the models may remain \textit{practically} non-identifiable if the associated eigenvalue is dominant (in which case the solution converges back to the non-identifiable condition), and potentially otherwise if the solution is not observed to diverge sufficiently.

On the other hand, the scalar nonlinear parabolic equation we study is unconditionally identifiable, subject to a spatially inhomogeneous initial condition. Established formally in \cref{sec:inhomogeneous_pde}, this result can also be viewed as a direct consequence of multi-experiment identifiability in linear equations by considering a suite of linearisations of the nonlinear model. Each ``instance'', say $i$, will yield an (identifiable) eigenvalue that characterises $\mathcal{A}_i$, the intersection of each being empty, thereby indicating that the model is unconditionally identifiable. Practically, however, parameter points may be indistinguishable when experimental noise is considered, particularly if the eigenvalues are similar; in the logistic model, for instance, this will be the case if the carrying capacity is sufficiently high. Identifiability for linear systems, on the other hand, corresponds to results for linear scalar systems. As we demonstrate, particular choices for the initial condition in a two-species cell motility model may result in indistinguishability: for such (hypothetical) experiments, it is impossible to distinguish between diffusion, growth, and exchange of the populations. In particular, we observe that there exists a continuous spectrum of solutions parametrised by $(\delta_1, \delta_2) \in \mathbb{R}^2$. 

Overall, we present a new perspective to identifiability analysis of a class of relatively simple PDE models that draws on well-established techniques for more standard analysis of PDEs. Importantly, our direct approach is relatively accessible and not only establishes identifiability, but provides critical information about how identifiability can be lost for particular choices of boundary and initial conditions. We have restricted the present work to cases where the system is fully observed and where the initial and boundary conditions are specified and parameter-independent. These scenarios lead to somewhat pathological losses of identifiability that nevertheless have important practical ramifications (see \cref{sec:practical}). For scalar linear parabolic equations, our description is relatively complete; however, there are many clear avenues left open for future consideration that we detail below.

\subsection{Future directions}

First, we have consistently made the idealistic assumption that one has complete (i.e., spatially and temporally continuous) observation of the solution profile and that parameters are constant. In practice, this is not often the case. It is, therefore, natural to consider scenarios where data is only observed at discrete coordinates in space or time, especially for cases where we do not have observations that are fine-grain enough in time to conclude that $u_t \not \equiv 0$. For example, there may be scenarios where the spatial profile is observed to be approximately constant, but where the solution behaviour changes rapidly between observation coordinates. This introduces additional challenges even if the spectrum of the auxiliary elliptic problem is discrete. Similarly, experiments often incorporate spatial heterogeneity in the environment or substrate, such that parameters (known or unknown) can be viewed as functions of space or time; for example, spatial variation in antibiotic concentration in \cite{Baym.2016}. Our approach is applicable to such spatially inhomogeneous problems, at least when the heterogeneity appears at a lower order (e.g., in growth for linear equations or in the carrying capacity for reaction-diffusion equations). The technique remains unchanged, but care should be given to the different form of the auxiliary eigenvalue problem one obtains.

Secondly, we have assumed that the boundary and initial conditions are known and parameter-independent. In some cases, the boundary operator $\mathcal{B}$ could depend parameters, as analysed in \cite{Renardy.2022}. Should the boundary condition also remain linear in parameters, one could directly apply our methodology, though we save such considerations for future work.

Thirdly, we have explored briefly how a nonlinear reaction term can lead to unconditional global identifiability. In contrast with homogeneous linear equations, this is preferred from a practical point of view. In particular, it may be possible to generalise our results to a wider class of nonlinear reaction-diffusion equations, which we expect to yield unconditional identifiability in cases that the initial condition is inhomogeneous.

Finally, we have briefly touched on implications for systems of equations, particularly those which are of reaction-diffusion type. This is primarily to keep exposition at a reasonable length; further development of this approach to more general reaction-diffusion systems (e.g., competitive, cooperative, or predator-prey models) would be an interesting further development. This approach should allow one to fully classify initial data from which two-species systems may lose identifiability. The complexity of our direct approach will, however, increase with the number of parameters which (even if still linear) will increase the number of cases to consider. We leave such considerations for future work.


\section*{Data availability}

Code used to produce the results are available at \url{https://github.com/ap-browning/parabolic_pde_identifiability}.

\section*{Acknowledgements}

YS was supported by a NSERC Postdoctoral Fellowship (NSERC Grant PDF-578181-2023). APB thanks the Mathematical Institute, University of Oxford, for a Hooke Research Fellowship. The authors thank Torkel Loman, Domenic Germano, and the anonymous referee, for helpful comments.

 
\end{document}